\documentclass[12pt,centertags,oneside]{amsart}
\usepackage{amsmath,amstext,amsthm,amscd,typearea,hyperref}
\usepackage{amssymb}
\usepackage{a4wide}
\usepackage[mathscr]{eucal}
\usepackage{mathrsfs}
\usepackage{typearea}
\usepackage{charter}
\usepackage{pdfsync}
\usepackage[a4paper,width=16.2cm,top=3cm,bottom=3cm]{geometry}

\numberwithin{equation}{section}

\newtheorem{theorem}{Theorem}[section]

\newtheorem{lemma}[theorem]{Lemma}
\newtheorem{corollary}[theorem]{Corollary}
\theoremstyle{definition}
\newtheorem{definition}[theorem]{Definition}
\newtheorem{example}[theorem]{Example}

\newcommand{\cali}[1]{\mathscr{#1}}

\newcommand{\Div}{{\rm Div}}

\newcommand{\supp}{{\rm supp}}

\newcommand{\ddc}{{dd^c}}

\newcommand{\dbar}{{\overline\partial}}

\newcommand{\id}{{\rm id}}

\newcommand{\reg}{{\rm reg}}

\newcommand{\met}{{\rm Met}}
\newcommand{\eq}{{\rm eq}}

\newcommand{\Cc}{\cali{C}}

\newcommand{\FS}{{{_\mathrm{FS}}}}
\newcommand{\MP}{{{_\mathrm{MP}}}}

\newcommand{\C}{\mathbb{C}}

\newcommand{\N}{\mathbb{N}}

\newcommand{\R}{\mathbb{R}}

\renewcommand{\P}{\mathbb{P}}

\newcommand{\cO}{{\mathcal O}}

\title[Equidistribution and convergence speed for zeros of 
holomorphic sections]{Equidistribution and convergence speed 
for zeros of holomorphic sections of singular Hermitian line bundles}
\author{Tien-Cuong Dinh}
\address{Department of Mathematics, National University 
of Singapore, 10 Lower Kent Ridge Road, Singapore 119076}
\email{matdtc@nus.edu.sg,  
url : http://www.math.nus.edu.sg/$\sim$matdtc}
\thanks{T.-C.\ D.\ partially supported by Start-Up 
Grant R-146-000-204-133 from National University of Singapore}
\author{Xiaonan Ma}
\address{Institut Universitaire de France
\&Universit\'e Paris Diderot - Paris 7,
UFR de Math\'ematiques, Case 7012,
75205 Paris Cedex 13, France}
\email{xiaonan.ma@imj-prg.fr}
\thanks{X.\ M.\ partially supported by 
Institut Universitaire de France and 
funded through the Institutional Strategy of 
the University of Cologne within the German Excellence Initiative}
\author{George Marinescu}
\address{Universit{\"a}t zu K{\"o}ln,
Mathematisches Institut, Weyertal 86-90, 50931 K{\"o}ln, Germany\\
    \& Institute of Mathematics `Simion Stoilow', Romanian Academy,
Bucharest, Romania}
\email{gmarines@math.uni-koeln.de}
\thanks{G.\ M.\ partially supported by DFG funded 
projects SFB/TR 12, MA 2469/2-1 and ENS Paris}

\date{\today}

\begin{document}

\begin{abstract}
We establish the equidistribution of zeros of random holomorphic 
sections of powers of a semipositive singular Hermitian line bundle, 
with an estimate of the convergence speed.
\end{abstract}

\maketitle

\tableofcontents
%%%%%%%%%%%%%%%%%%%%%%%%%%%%%%%%%%%
\section{Introduction}\label{s0}
%%%%%%%%%%%%%%%%%%%%%%%%%%%%%%
The purpose of this paper is to study the convergence speed of
the zero-divisors of random sequences of holomorphic sections 
in high tensor powers of a 
holomorphic line bundle endowed with a singular Hermitian metric.
 
Distribution of zeros of random polynomials is a classical subject, 
starting with the papers of Bloch-P\'olya, 
Littlewood-Offord, Hammersley, Kac and Erd\"os-Tur\'an, 
see e.g., \cite{BDi:97,BLev,BSh,SheppVanderbei} 
for a review and complete references.
After the work of Nonnenmacher-Voros %in dimension one
\cite{No:97,NoVo:98}, general methods were developed by 
Shiffman-Zelditch \cite{ShZ99} and Dinh-Sibony \cite{DS06b} 
to describe the asymptotic distribution of zeros of random 
holomorphic sections of a positive line bundle over 
a projective manifold endowed with a smooth positively 
curved metric. 
The paper \cite{DS06b} gives moreover a good estimate of
the convergence speed and applies to general measures 
(e.\,g.,\ equidistribution of complex zeros of homogeneous 
polynomials with real coefficients). 
These methods were extended to the non-compact setting
in \cite{DMS}. Some important technical tools for higher dimension 
used in the previous works were introduced
by Forn{\ae}ss-Sibony \cite{FS95}.

In \cite{CM11} it was shown that the equidistribution results
from \cite{DS06b,ShZ99} extend to the case of a singular 
Hermitian holomorphic line bundle with strictly positive curvature 
current. 
%We will combine the results and techniques from
%\cite{DS06b} and \cite{CM11} to obtain information about the 
%convergence speed in the case of singular metrics. 

We will start with an abstract statement. For an arbitrary complex
vector space $V$ 
we denote by $\mathbb{P}(V)$ the projective space of 
$1$-dimensional subspaces of $V$. For $v\in V$ we denote by
$[v]$ its class in $\mathbb{P}(V)$. Fix now a vector space $V$ 
of complex dimension $d+1$.
Recall that there is a canonical identification of 
$\mathbb{P}(V^{*})$ with the Grassmannian 
$G_{d}(V)$ of hyperplanes in $\P(V)$, given by 
$\mathbb{P}(V^{*})\ni[\xi]\longmapsto H_{\xi}
:=\P(\ker\xi)\in G_{d}(V) $, 
for $\xi\in V^{*}\setminus\{0\}$.
If $V$ is endowed with a Hermitian metric, then we denote by 
$\omega_\FS$ the induced Fubini-Study 
form on projective spaces $\P(V)$ normalized so that 
$\sigma_\FS:=\omega_\FS^d$ is a probability measure.
We also use the same notations for $\P(V^*)$.

Fix an integer $1\leqslant k\leqslant n$. 
%We want to study the distribution of the common zeros 
%of $k$ sections, i.e., the set
%$$\{s^{(1)}=\cdots=s^{(k)}=0\}.$$
We consider on $\P(V^*)^k$ the Haar measure $\sigma_\MP$ 
associated with the natural action of 
the unitary group on the factors of $\P(V^*)^k$ (cf. \eqref{MP}).
If $\xi=(\xi_1,\ldots,\xi_k)$ is a point in $\P(V^*)^k$, denote by 
$H_\xi$ the intersection of the hyperplanes $H_{\xi_i}$ in $\P(V)$. 
The following extension of Theorem \ref{th_abstract_1} 
stated below
can be obtained using the ideas in  \cite{DMS} and \cite{DS06b}.
This is a version of Large Deviation Theorem in our setting.
\begin{theorem} \label{th_abstract_high}
%Let $\Phi$ be as in Theorem \ref{th_abstract_1}. 
Let $(X,\omega_X)$ be a compact K\"ahler manifold of dimension 
$n$ and let $V$ be a Hermitian 
complex vector space of dimension $d+1$. 
Let $\Phi:~X\dashrightarrow~\P(V)$ be a meromorphic map. 
Then, there exist $c>0$ depending only on $(X,\omega_X)$ 
and $m>0$ depending only on $k$ such that for any $\gamma>0$ 
there is a subset $E_{\gamma}$ of $\P(V^*)^k$ 
with the following properties
\begin{enumerate}
\item[(a)] $\sigma_\MP(E_{\gamma})
\leqslant c\, d^m e^{-\gamma/c}$. 
\item[(b)] For $\xi$ outside $E_{\gamma}$\,, 
the current $\Phi^*[H_\xi]$ is well-defined and 
\begin{align}\label{eq:1.1}
\|\Phi^*[H_\xi]-\Phi^*(\omega_\FS^k)\|_{-2}
\leqslant \gamma m_{k-1},
\end{align}
where $m_k$ denotes the mass of the current 
$\Phi^*(\omega_\FS^k)$ {\rm(}cf. (\ref{eq:3.1}), (\ref{eq:3.4})
for the definitions of the semi-norm $\|\cdot\|_{-2}$
and mass on currents{\rm)}. 
\end{enumerate}
\end{theorem}

Consider now a holomorphic line bundle $L\to X$ 
on a compact K\"ahler manifold $(X,\omega_{X})$
endowed with a singular Hermitian metric $h^L$.
Let $K_{X}$ be the canonical line bundle on $X$
with the metric induced by $\omega_{X}$.
Let $(F,h^F)$ be an auxiliary Hermitian holomorphic line bundle 
endowed with a smooth metric $h^F$. 
These metrics and the volume form $\omega_X^n$ induce 
an $L^2$ scalar product 
\eqref{herm_prod} on the space of sections of $L^p\otimes F$ 
and we denote by $H^0_{(2)}(X,L^p\otimes F)$
the space of holomorphic $L^2$ sections (cf.\ \eqref{ell2}). 
These spaces are finite dimensional Hilbert spaces 
endowed with the scalar product \eqref{herm_prod}.

This induces Fubini-Study metrics $\omega_{\FS}$ and probability 
measures
$\sigma_\FS$ on the spaces $\P(H^0_{(2)}(X,L^p\otimes F))$
and also multi-projective metrics 
$\omega_\MP$ and natural probability measures 
$\sigma_p:=\sigma_{p,\MP}$
on $\P(H_{(2)}^0(X,L^p\otimes F))^k$ (see \eqref{MP}).
Consider the probability space
\begin{equation}\label{probsp_2}
(\Omega_k(L,F),\sigma_{\infty})
:=\prod_{p=1}^\infty \big(\P (H_{(2)}^0(X,L^p\otimes F))^k,
\sigma_p\big).
\end{equation}
Although we don't indicate explicitly, these spaces depend on 
$h^L$, $h^F$. 
If $F$ is trivial we just write $(\Omega_k(L),\sigma_{\infty})$.

%We consider the probability space
%\begin{equation}\label{probsp}
%(\Omega,\sigma_\infty):=\prod_{p=1}^\infty 
%(\P(H^0_{(2)}(X,L^p\otimes F)),\sigma_{p})\,.
%\end{equation}
We have the following equidistribution result with speed estimate 
for the zeros of random 
$L^2$ holomorphic sections of big line bundles endowed with
semipositively curved metrics. For a holomorphic section $s$ 
of a line bundle we denote by $\Div(s)$ the associated divisor
and by $[\Div(s)]$ the current of integration on $\Div(s)$.
We refer to Definition
\ref{def_speed} for the notion of convergence speed of currents. 
%-------------------
\begin{theorem}\label{th_eq_speed}
Let $(X,\omega_X)$ be a compact K\"ahler manifold of dimension 
$n$, and let $L$ be a holomorphic line bundle endowed with 
a singular metric $h^L$ such that $c_1(L,h^L)\geq 0$ on $X$.

(i) Assume that $L$ is big and let $\widetilde{h}^L$ be 
a singular Hermitian metric on $L$ with 
$c_1(L,\widetilde{h}^L)\geq \varepsilon\omega_X$ 
for some $\varepsilon>0$. 
Assume that $h^L\leqslant A\, \widetilde{h}^L$ for some 
constant $A> 0$. 
%Let $\alpha=\alpha(X,L)$ be the given by 
%Corollary \ref{cor_abstract} and let $(\lambda_p)$ 
%be a sequence of positive numbers satisfying 
%\begin{equation}\label{e:abs1.3a}
%\liminf_{p\to\infty} \frac{\lambda_p}{\log p}>(2n+1)\alpha\,,\;\;
%\lim_{p\to\infty} \frac{\lambda_p}{p}=0\,.
%\end{equation}
Then for $\sigma_\infty$-almost every sequence 
$([s_p])\in(\Omega_1(L),\sigma_{\infty})$,
%$([s_p])\in\prod_{p=1}^\infty \P\big(H^0_{(2)}(X,L^p)\big)$, 
${(\frac1p}[\Div(s_p)])$ 
converges to $c_1(L,h^L)$ on $X$ as $p\to\infty$ with 
speed $O\big(\frac1p\log p\big)$.

(ii) Let $U\subset X$ be an %relatively compact 
open set such that 
$c_1(L,h^L)\geq \varepsilon\omega_X$ on a neighborhood of 
$\overline U$ for some $\varepsilon>0$. 
Then for $\sigma_\infty$-almost every sequence 
$([s_p])\in(\Omega_1(L,K_X),\sigma_\infty)$,
%$([s_p])\in\prod_{p=1}^\infty 
%\P\big(H^0_{(2)}(X,L^p\otimes K_X)\big)$, 
${(\frac1p}[\Div(s_p)])$ 
converges to $c_1(L,h^L)$ on $U$ as $p\to\infty$ 
with speed $O\big(\frac1p\log p\big)$.
\end{theorem}
%-------------------
The assumption $h^L\leqslant A\, \widetilde{h}^L$ in (i) means that
$h^L$ is less singular than the positively curved metric
 $\widetilde{h}^L$. 
Note that the assumptions in (i) and (ii) are necessary. 
Without them there could be very few sections in $H^0(X,L^p)$ 
or $H^0(X,L^p\otimes K_X)$, respectively, that is, 
their dimension could be bounded independently of $p$.

We consider next continuous Hermitian metrics on 
ample line bundles. Let $L$ be an ample line bundle 
over a compact K\"ahler manifold $X$ of dimension $n$. 
Let $h^L_{0}$ be a smooth Hermitian metric 
on $L$ such that $\alpha=c_{1}(L,h^L_{0})$ is a K\"ahler form.
%Assume that $L$ is ample. So, we can assume that $\alpha$ is 
%be a K\"ahler form.
Let $h^L$ be a continuous Hermitian metric on $L$ which
is associated with a continuous function 
$\varphi$ by $h^L=h^L_{0}e^{-2\varphi}$. We call $\varphi$ 
a global weight of $h$.
%Assume that $\varphi$ is H\"older continuous in a neighbourhood 
%of an open set $U\subset X$. 
We do not assume that the curvature current $c_{1}(L,h^L)$ 
is positive (it is not of order 0 in general).  

Define the equilibrium weight $\varphi_{\eq}$ associated with the 
continuous weight $\varphi$ as the upper envelope of all
$\alpha$-psh functions (cf.\ \eqref{alpha-psh}) smaller 
than $\varphi$ on $X$, 
\begin{equation}\label{e_phie}
\varphi_{\eq}:X\to[-\infty,\infty),\:\:\varphi_{\eq}(x)
:={\sup}^{*}\Big\{\psi(x):  
\text{$\psi\in PSH(X,\alpha)$, $\psi\leq\varphi$\:\:on $X$}\Big\}
\end{equation}
where the star denotes upper semi-continuous regularization. 
(The upper semi-continuous regularization of a function $\psi$ is 
$\psi^*(x)=\limsup_{y\to x}\psi(y)$.)
The equilibrium first Chern form is defined by
\begin{align}\label{eq:1.4}
\omega_{\eq}:=\alpha+\ddc \varphi_{\eq}\,.
\end{align}
The equilibrium metric on $L$ is given by 
$h^L_{\eq}=h^L_{0}e^{-2\varphi_{\eq}}$; it satisfies
$c_1(L,h^L_{\eq})=\omega_{\eq}$. 
Note that here $\varphi_{\eq}$ is bounded on $X$ 
since $\varphi$ is bounded and $\alpha$ is a K\"ahler form, 
so constant functions are $\alpha$-psh.
Therefore, the wedge-products $\omega_{\eq}^{k}$, 
$1\leq k\leq n$, are well-defined \cite{BT82}.
%on the set where $\varphi_{\eq}$ is locally bounded \cite{BT82}.
The equilibrium measure is given by $\mu_{\eq}=\omega_{\eq}^{n}$.
When $X$ is the projective line $\P^{1}$ and $L$ is the hyperplane
line bundle $\mathcal{O}(1)$, the measure $\mu_{\eq}$ is 
a minimizer of the weighted logarithmic energy \cite{ST97}.

The following result generalizes a result by Berman \cite{Berman} 
where smooth weights $\varphi$ were considered.
It shows that the equilibrium weight
of a global H\"older weight can be uniformly approximated by 
global Fubini-Study weights, with speed estimate.
%-----------------
\begin{theorem} \label{th_cv_holder}
Let $(X,\omega_X)$ be a compact K\"ahler manifold, $(L,h^L_{0})$ 
be an ample line bundle
endowed with a smooth metric $h^L_{0}$ such that 
$c_{1}(L,h^L_{0})$ is a K\"ahler form.
Let $h^L=h^L_{0}e^{-2\varphi}$ be a singular metric on $L$, 
such that $\varphi$ is %a measurable bounded function on $X$ and
H\"older continuous on $X$.
%in a neighborhood of an open set $U\subset X$. 
Then the equilibrium weight $\varphi_{\eq}$ is 
continuous on $X$. Moreover, the global Fubini-Study
weights $\varphi_p$ given by \eqref{e:21} converge to 
$\varphi_{\eq}$ with estimate
\begin{equation}\label{e5.1}
\|\varphi_p-\varphi_{\eq}\|_{\infty}=O
\left(\frac1p{\log p}\right)\,,\:\:p\to\infty\,,
\end{equation}
where $\|\cdot\|_{\infty}$ denotes the supremum norm on $X$.
In particular, for any $1\leq k\leq n$ we have 
$\frac{1}{p^k}\omega_p^k\to\omega_{\eq}^k$ on $X$ as 
$p\to\infty$ with speed $O\big(\frac1p\log p\big)$. 
\end{theorem}
%-----------------
\begin{corollary}\label{cor_cv_holder}
Let $(X,\omega_X)$, $(L,h^L)$  be as in 
Theorem \ref{th_cv_holder}. Let $1\leq k\leq n$.  
Then for $\sigma_{\infty}$-almost every sequence 
$(S_p)\in(\Omega_k(L),\sigma_\infty)$,
$S_p=([s_p^{(1)}],\ldots,[s_p^{(k)}])$,
the sequence of currents of integration on the common zeros 
$\frac{1}{p^k}\big[s_p^{(1)}=\ldots=s_p^{(k)}=0\big]$ converges 
to $\omega_{\eq}^k$ on $X$ as $p\to\infty$ 
with speed $O\big(\frac1p\log p\big)$. 
\end{corollary}
%-----------------
The paper is organized like follows. In Section \ref{s:prelim} 
we recall the notions of Bergman kernel and Fubini-Study currents
in the context of singular Hermitian metrics. 
In Section \ref{section_abstract}
we describe a general setting for the equidistribution of zeros,
which also delivers precise information 
about the convergence speed. In Section \ref{section_hyp} 
we apply these results to semipositive
Hermitian metrics and prove Theorem \ref{th_eq_speed}.
Finally, in Section \ref{s:hoelder} we consider the case of 
arbitrary singular metrics and prove Theorem \ref{th_cv_holder}
and Corollary \ref{cor_cv_holder}. 

\medskip
\noindent
\textbf{Acknowledgements.} The paper was partially written during 
the visits of the first author at the 
Laboratorio Fibonacci and Centro De Giorgi
and the second author at National University of Singapore. 
They would like to thank 
these organizations and Marco Abate, Carlo Carminati, 
Stefano Marmi, David Sauzin for their hospitality.
%during his stay in Pisa. 
%X. M. likes to thank National University of Singapore
%for the hospitality.
%%%%%%%%%%%%%%%%%%%%%%%%%%%%%%%%%
\section{Preliminaries}\label{s:prelim}
%%%%%%%%%%%%%%%%%%%%%%%%%%%%%%%%%
Let $X$ be a complex manifold.
We assume that the reader is acquainted with the notion of
plurisubharmonic (henceforth abbreviated psh) function 
$\varphi:X\to[-\infty,\infty)$, see \cite[Ch.\,I\,(5.1)]{D12}.
Recall that psh functions are locally integrable 
(\cite[Ch.\,I\,(4.17), (5.3)]{D12}).
A function $\varphi:X\to[-\infty,\infty)$ is called quasi-psh if it is 
locally given as the sum of a psh function and a smooth function.

We also assume that the reader is familiar to the notion of 
positive current (in the sense of Lelong, i.\,e., non-negative, 
see \cite[Ch.\,III\,(1.13)]{D12}, \cite[B.2.11]{MM07}).
For a positive current $\beta$ we write $\beta\geq0$.
If $\alpha$ is a closed real current of bidegree $(1, 1)$ on $X$
we define the space of $\alpha$-psh functions as
\begin{equation}\label{alpha-psh}
PSH(X,\alpha):=\big\{\varphi:X\to[-\infty,\infty): \text{$\varphi$ 
quasi-psh,
$\ddc\psi+\alpha\geq0$}\big\}.
\end{equation}
Here $d^c=\frac{1}{2\pi i}(\partial-\overline\partial)$, hence 
$dd^c=\frac{i}{\pi}\partial\overline\partial$.

Let $(X,\omega_X)$ be a compact K\"ahler manifold of dimension
$n$ and consider a holomorphic line bundle $L\to X$. 
Let $U\subset X$ be an open set for which there exists 
a local holomorphic frame $e_{L}:U\longrightarrow L$. 

\par Let $h^L$ be a smooth Hermitian metric on $L$. 
Recall that the first Chern %curvature 
form $c_1(L,h^L)$ of $h$ is defined by
\begin{equation}\label{e:curb1}
c_1(L,h^L)\mid_{_{U}}=-dd^c\log|e_{L}|_{h^L}=\frac{i}{2\pi}\,R^L,
\end{equation} 
where $R^L$ is the curvature of the holomorphic Hermitian 
connection $\nabla^L$ on $(L,{h^L})$.

\par If $h^L$ is a singular Hermitian metric on $L$ then 
we set 
\begin{equation}\label{e:lw}
|e_{L}|_{h^L}^2=e^{-2\varphi},
\end{equation} 
where the function $\varphi\in L^1_{loc}(U)$ is called the 
local weight of the metric $h$ with respect to the frame $e_{L}$ 
(see \cite{D90}, also \cite[p.\ 97]{MM07}). 
The curvature of ${h^L}$, 
\begin{equation}\label{e:curb2}
c_1(L,h^L)\mid_{_{U}}=dd^c\varphi\,,
\end{equation}
is a well-defined closed (1,1) current on $X$. 
The cohomology class of $c_1(L,h^L)$ in $H^{1,1}(X,\R)$ does not 
depend on the choice of ${h^L}$.
This is the Chern class of $L$ and we denote it by $c_1(L)$. 

We say that the metric ${h^L}$ is 
semipositively curved if $c_1(L,h^L)$ is a positive current. 
Equivalently, the local weights $\varphi$ given by \eqref{e:lw}  
are (equal almost everywhere) to psh functions. 
Recall that a line bundle $L$ is said to be pseudoeffective if
it admits a (singular) semipositively curved metric $h^L$ 
(see \cite{D90}).

Let $L$ be a holomorphic line bundle and $h^L_0$ be 
a smooth metric on $L$. Set $\alpha=c_1(L,h^L_0)$. 
Let us denote by $\met^{+}(L)$ the set of semipositively 
curved metrics on $L$. There exists a bijection
\begin{equation}\label{psh-met}
PSH(X,\alpha)\longrightarrow\met^+(L)\,,\:\: 
\varphi\longmapsto h^L_{\varphi}=h^L_0 e^{-2\varphi},
\end{equation}
and $c_1(L,h^L_{\varphi})=\alpha+dd^c\varphi$.

Let $(F,h^F)$ be an auxiliary Hermitian holomorphic line bundle
endowed with a smooth metric $h^F$. We denote by
\begin{equation}\label{e:hp}
h_p=(h^L)^{\otimes p}\otimes h^F,
\end{equation} 
the metric induced by $h^L$, $h^F$ on $L^p\otimes F$.
Consider the space $L^2(X,L^p\otimes F)$ of $L^2$ sections 
of $L^p\otimes F$ relative to the metric $h_p$ and the volume form  
$\omega_X^n$ on $X$, endowed with the inner product
\begin{equation}\label{herm_prod}
(s,s')_p=\int_{X}\langle s,s'\rangle_{h_p}\,\omega_X^n\,,\;\;
{\rm where}\;\;s,s'\in L^2(X,L^p\otimes F).
\end{equation} 
We let $\|s\|_p^2=(s,s)_p$. 
Let us denote by
\begin{equation}\label{ell2}
H^0_{(2)}(X,L^p\otimes F):=\Big\{s\in L^2(X,L^p\otimes F): 
s \text{ holomorphic}\Big\}
\end{equation}
the space of $L^2$-holomorphic sections of $L^p\otimes F$.
In the same way, let $L^2_{q,r}(X,L^p\otimes F)$
%(resp. $L^2_{q,r}(X,L^p\otimes F,loc)$ 
be the space of %(resp. locally) 
$L^2$-integrable $(q,r)$-forms with values in 
$L^p\otimes F$ relative to $h_{p}$  and $\omega_{X}$.
We will add `loc' for spaces of locally $L^2$-integrable forms 
when  $X$ is not compact.

For a section $s\in H^0_{(2)}(X,L^p\otimes F)$ we denote by 
$\Div(s)$ the divisor defined by $s$ (cf.\ \cite[(2.1.4)]{MM07}) and 
by $[\Div(s)]$ the current of integration on $\Div(s)$ 
(cf.\ \cite[Ch.\,III\,(2.5)]{D12}, \cite[(B.2.16)]{MM07}).
Note that for two non-zero elements 
$s,s'\in H^0_{(2)}(X,L^p\otimes F)$ 
which are in the same equivalence class 
in $\P(H^0_{(2)}(X,L^p\otimes F))$ we have 
$\Div(s)=\Div(s')$, so $\Div$ is well-defined 
on $\P(H^0_{(2)}(X,L^p\otimes F))$.

Assume now that $(L,{h^L})$ 
is a holomorphic line bundle endowed 
with semipositively curved singular metric. 
Denote by $\Sigma\subset X$ the set of points where ${h^L}$
is not bounded. This set has zero Lebesgue mass. 
%We define the spaces $L^2(X,L^p)$ and $H^0_{(2)}(X,L^p)$ 
%as in \eqref{herm_prod}-\eqref{ell2}.
Let $\{s^p_j\}_{j=1}^{d_p}$ be an orthonormal basis of 
$H^0_{(2)}(X,L^p\otimes F)$.  
Let $B_p$ be the Bergman kernel function defined by 
\begin{equation}\label{e:Bergfcn}
B_p(x)=\sum_{j=1}^{d_p}|s^p_j(x)|_{h_{p}}^2\,,\;\;\;
x\in X\setminus\Sigma\, ,
\end{equation}
where $h_p$ is given by \eqref{e:hp}. 
Let $s^p_j=f^p_je_L^{\otimes p}\otimes e_F$, where 
$f^p_j\in\mathcal{O}(U)$ and $e_L$, $e_F$
are holomorphic frames of $L$, $F$ on $U$. 
Let $\varphi'$ be the local weight of
$h^F$ with respect to $e_F$, defined as in \eqref{e:lw}.
Then on $U\setminus\Sigma$ the following holds
\begin{equation}\label{e:Bergfcn1}
\log B_p=\log\Big(\sum_{j=1}^{d_p}|f^p_j|^2\Big)
-2p\,\varphi-2\,\varphi'\,.
\end{equation}
The right-hand side of \eqref{e:Bergfcn1} is a difference of psh 
(hence locally integrable) functions on $U$, 
so defines an element in $L^1_{loc}(U,\omega_X^n)$. 
Therefore, $\log B_p$ defines an element in $L^1(X,\omega_X^n)$. 

%\noindent
The Kodaira map is the meromorphic map given by
\begin{equation}\label{e:kod}
\begin{split}
&\Phi_p:X\dashrightarrow
\P\big(H^0_{(2)}(X,L^p\otimes F)^*\big)\,,\\
\Phi_p(x)=&\big\{s\in H^0_{(2)}(X,L^p\otimes F):
s(x)=0\big\}\,,\:\:x\in X\setminus Bs_p\,,
\end{split}
\end{equation}
where a point in $\P\big(H^0_{(2)}(X,L^p\otimes F)^*\big)$
is identified with a hyperplane through the origin
in $H^0_{(2)}(X,L^p\otimes F)$ and $Bs_p=\{x\in X:s(x)=0\;
\text{for all $s\in H^0_{(2)}(X,L^p\otimes F)$}\}$ 
is the base locus of $H^0_{(2)}(X,L^p\otimes F)$. 
We define the \emph{Fubini-Study currents} by 
\begin{equation}\label{fs}
\omega_p=\Phi_p^*(\omega_{_\mathrm{FS}}),
\end{equation}
where $\omega_\FS$ denotes the Fubini-Study $(1,1)$-form on 
$\P\big(H^0_{(2)}(X,L^p\otimes F)^*\big)$.
They are positive closed $(1,1)$-currents obtained by pulling back
the Fubini-Study form $\omega_\FS$. 
The current $\omega_p$ is in fact given by an $L^1$-form on $X$, 
which is smooth outside the set of indeterminacy of $\Phi_p$, 
see Lemma \ref{lemma_pullback} below. We have
\begin{equation}\label{fs1}
\omega_p\mid_{_U}
=\frac12\,\ddc\log\Big(\sum_{j=1}^{d_p}|f^p_j|^2\Big)\,,
\end{equation}
hence by \eqref{e:Bergfcn1}
\begin{equation}\label{e:Bergfcn2}
\frac12\,\ddc\log B_p=\omega_p-p\,c_1(L,h^L)-\,c_1(F,h^F)\,.
\end{equation}

We used above the following basic property that we will give 
a proof for the reader's convenience. 

\begin{lemma} \label{lemma_pullback}
Let $\Phi:Y\dashrightarrow Z$ be a meromorphic map between two compact 
complex manifolds $Y$, $Z$ 
of dimensions $\ell$ and $m$ respectively. 
Let $\alpha$ be a smooth $(q,r)$-form on $Z$ 
with $0\leq q,r\leq \min(\ell,m)$. 
Then the $(q,r)$-current $\Phi^*(\alpha)$ on $Y$ 
is well-defined and given by a $(q,r)$-form with $L^1$ coefficients 
which is smooth outside the indeterminacy set of $\Phi$. 
\end{lemma}
\proof
Recall that for a meromorphic map $\Phi:Y\dashrightarrow Z$
(\cite[Definition\,2.1.19]{MM07}, \cite{Re57}) 
there is an analytic subset $I$ of $Y$ such that $\Phi$ 
is holomorphic on $Y\setminus I$ and the closure of the graph 
of $\Phi$ over $Y\setminus I$ is an irreducible analytic subset 
of dimension $\ell$ of $Y\times Z$, called the graph of $\Phi$. 
The smallest set $I$ with this property is called the indeterminacy 
set of $\Phi$. Since $Y$ is a manifold, 
$I$ is of codimension at least two \cite[p.\,333]{Re57}. 
Denote by $\Gamma$ the graph of $\Phi$. 
It defines, by integration on its regular part $\reg(\Gamma)$, 
a positive closed current $[\Gamma]$ of bi-dimension $(\ell,\ell)$ 
in $Y\times Z$ \cite[p.\,140]{D12}.

Denote by $\pi_Y,\pi_Z$ the natural projections from $Y\times Z$
to $Y$ and $Z$ respectively. 
The pull-back $\Phi^*(\alpha)$ is defined by 
\begin{equation}\label{e:ur}
\Phi^*(\alpha):= (\pi_Y)_*(\pi_Z^*(\alpha)\wedge [\Gamma]).
\end{equation}
This is the formal definition for any current $\alpha$. 
It makes sense when the wedge-product 
in the last expression is well-defined because here the operator 
$(\pi_Y)_*$ is well-defined on all currents. 
In our setting, since $\pi_Z^*(\alpha)$ is smooth, the current 
$\Phi^*(\alpha)$ is well-defined. 
More precisely, if $\beta$ is a smooth form of bidegree 
$(\ell-q,\ell-r)$ on $Y$ then 
\begin{equation}\label{e:ur1}
\langle \Phi^*(\alpha),\beta\rangle 
= \int_{\reg(\Gamma)} \pi_Z^*(\alpha)\wedge \pi_Y^*(\beta).
\end{equation}
Note that the $2\ell$-dimensional volume of $\Gamma$ 
is finite \cite[p.\,140]{D12}. %by Wirtinger's theorem.

Formula \eqref{e:ur1} shows that the current $\Phi^*(\alpha)$ 
extends continuously to the space of test 
forms $\beta$ with continuous coefficients. So $\Phi^*(\alpha)$
is a current of order 0. 
If $V$ is a proper analytic subset of $Y$,
then $\Gamma\cap\pi_Y^{-1}(V)$ is a proper analytic subset
of $\Gamma$, so $\Gamma\cap \pi_Y^{-1} (y)$ has 
zero $2\ell$-dimensional volume.
Therefore, the last formula implies 
that $\Phi^*(\alpha)$ has no mass on $V$, in particular, 
this current has no mass on the indeterminacy set $I$.

If $\beta$ has compact support in $Y\setminus I$,
since $\pi_Y$ defines a bi-holomorphic map from 
$\Gamma\setminus \pi_Y^{-1}(I)$ to $Y\setminus I$, 
the last integral is equal to  the integral on $Y\setminus I$ 
of the form $(\pi_Y)_*(\pi_Z)^*(\alpha)\wedge\beta$. 
The last expression is equal to 
$(\Phi|_{Y\setminus I})^*(\alpha)\wedge \beta$, where 
$(\Phi|_{Y\setminus I})^*(\alpha)$ 
is the pull-back of the smooth form $\alpha$ by the holomorphic 
map $\Phi|_{Y\setminus I}$. 
We conclude that the current $\Phi^*(\alpha)$ is equal on 
$Y\setminus I$ to the smooth form 
$(\Phi|_{Y\setminus I})^*(\alpha)$. 
Finally, since $\Phi^*(\alpha)$ is of order 0 and has no mass 
on $I$, the form $(\Phi|_{Y\setminus I})^*(\alpha)$ 
has $L^1$ coefficients and is equal, 
in the sense of currents on $Y$, to  $\Phi^*(\alpha)$. 
This completes the proof of the lemma. 
\endproof

Note that the lemma can be extended to meromorphic maps
between open manifolds provided that $\pi_Y$ is proper 
on $\pi_Z^{-1}(\supp(\alpha))\cap\Gamma$. 
Moreover, by definition, if $\alpha$ is closed and/or positive 
then $\Phi^*(\alpha)$ is also closed and/or positive. 

%-------------------------------------------------
%%%%%%%%%%%%%%%%%%%%%%%%%%%%%%%%
\section{Abstract setting for equidistribution} 
\label{section_abstract}
%%%%%%%%%%%%%%%%%%%%%%%%%%%%%%%%%
We will only consider the case of compact K\"ahler manifolds 
but it is certainly easy to extend the 
results to the case of manifolds of Fujiki class and even open 
manifolds satisfying some properties of concavity.

Let $(X,\omega_X)$ be a compact K\"ahler manifold of 
dimension $n$. 
Recall that we can introduce several semi-norms on the set 
of currents of order 0 on $X$. 
If $U$ is an open subset of $X$, $\alpha$ is a strictly positive 
number and $T$ is a current of order 0 on $X$, define 
\begin{align}\label{eq:3.1}
\big\|T\big\|_{U,-\alpha}:=\sup \big|\langle T,u\rangle\big|
\end{align}
where the supremum is taken over smooth test forms $u$ with 
support in $U$ and such that 
their $\Cc^\alpha$-norm satisfies $\|u\|_{\Cc^\alpha}\leq 1$.  

For simplicity, we will drop the letter $U$ when $U=X$. 
In this case, $\|\cdot\|_{-\alpha}$ is a norm 
and the associated topology coincides with the weak topology 
on any set of currents with mass 
bounded by a fixed constant. We will only consider the case 
$\alpha=2$ and we will be interested 
in estimates on $\|\cdot\|_{U,-2}$. The other cases can 
be obtained as a consequence, e.g., if $\alpha<2$, 
we can use the theory of interpolation between Banach spaces
\cite{DS06b}, \cite{Triebel78}.

\begin{definition}\label{def_speed}
Let $(c_p)$ be a sequence of positive numbers converging to $0$. 
Let $\{T_p:p\in\N\}$ and $T$ 
be currents on $X$ with mass bounded by a fixed constant. 
We say that the sequence $(T_p)$ 
converges on $U$ to $T$ with speed $(c_p)$
if $\big\|T_p-T\big\|_{U,-2}\leqslant c_p$ for 
$p$ large enough. We also say that the sequence converges 
with speed $O(c_p)$ if 
it converges with speed $(Cc_p)$ for some $C\geq 0$.
\end{definition}

Recall that a current of order $0$ is an element in the dual of the 
space of continuous forms. The mass of such currents is 
the norm dual 
to the $\mathscr{C}^0$ norm on forms. However, for a positive 
$(q,q)$-current $T$ on $(X,\omega_{X})$,
it is more convenient to use the following notion of mass
\begin{align}\label{eq:3.4}
\|T\| = \left\langle  T, \omega_{X}^{n-q}\right\rangle
\end{align}
which is equivalent to the above mass-norm. The advantage is that 
when $T$ is positive closed, its mass only depends 
on its cohomology class in $H^{q,q}(X,\R)$.

The following result was obtained in  \cite[Theorem\,4]{DMS}, 
where we assumed 
that the map $\Phi$ has generically maximal rank $n$, 
but the proof there is valid 
without this condition. 

%------------------------------------------
\begin{theorem} \label{th_abstract_1}
Let $(X,\omega_X)$ be a compact K\"ahler manifold of 
dimension $n$ and let $V$ be a Hermitian 
complex vector space of dimension $d+1$. Consider a meromorphic
map $\Phi:~X\dashrightarrow  ~\P(V)$. 
Then there exists  $c>0$ depending only on $(X,\omega_X)$ 
such that for any $\gamma>0$ there is a subset $E_\gamma$ 
of $\P(V^*)$ with the following properties:
\begin{enumerate}
\item[(a)] $\sigma_\FS(E_\gamma)\leq c\,d^{\,2}\, e^{-\gamma/c}$. 
\item[(b)] For $\xi$ outside $E_\gamma$\,, the current 
$\Phi^*[H_\xi]$ is well-defined and 
\begin{align}\label{eq:3.2}
\big\|\Phi^*[H_\xi]-\Phi^*(\omega_\FS)\big\|_{-2}
\leqslant \gamma\,. 
\end{align}
\end{enumerate}
\end{theorem}
%------------------------------------------
Consider now holomorphic Hermitian line bundles $(L,h^L)$, 
$(F,h^F)$ such that $h^L$ is a singular Hermitian metric.
We have 
$H^0_{(2)}(X,L^p\otimes F)\subset H^0(X,L^p\otimes F)$, thus 
$d_p:=\dim H^0_{(2)}(X,L^p\otimes F)<\infty$.
We assume that $d_p\geq  1$. Note that there exists $C>0$ 
such that $d_p\leqslant Cp^n$ for all $p\in\N$, 
where $C>0$ is a constant depending only 
on $(X,\omega_X)$, $c_1(L)$, $c_1(F)$.
%$c_1(L,h^L)$ and $c_1(F,h^F)$. 
This follows from %Siegel's lemma \cite[Lemma\,2.2.6]{MM07} or 
the holomorphic Morse inequalities \cite[Theorem\,1.7.1]{MM07}
or the Siegel Lemma \cite[Lemma 2.2.6]{MM07}. 

%--------------------
We have the following consequence of the above result 
(compare also \cite[Theorem\,2]{DMS}).
%--------------------
\begin{corollary} \label{cor_abstract}
Let $(X,\omega_X)$ be a compact K\"ahler manifold of 
dimension $n$ and let $(L,h^L)$ be a singular 
Hermitian holomorphic line bundle on $X$. 
Let $(F,h^F)$ be a holomorphic line bundle 
with smooth Hermitian metric.
Then there is $c=c(X,L,F)>0$ depending only on $(X,\omega_X)$ 
and $c_1(L)$, $c_1(F)$, %$c_1(L,h^L)$, $c_1(F,h^F)$, 
with the following property. 
For any sequence of positive numbers $\lambda_p$\,, 
there are subsets $E_p\subset\P(H_{(2)}^0(X,L^p\otimes F))$ 
such that for $p$ large enough
\begin{equation}\label{e:abs1.0}
\sigma_p(E_p)\leqslant c\, p^{2n} e^{-\lambda_p/c}\,,
\end{equation}
\begin{equation}\label{e:abs1.1}
\big\|[\Div(s)] -\omega_p\big\|_{-2} \leqslant\lambda_p\,,\:\:
\text{for any 
$[s]\in \P(H^0_{(2)}(X,L^p\otimes F))\setminus E_p$}\,.
\end{equation}
Let $(\lambda_p)$ be a sequence of positive numbers such that
\begin{equation}\label{e:abs1.2}
\liminf_{p\to\infty} \frac{\lambda_p}{\log p}>(2n+1)c\,.
\end{equation}
Then for $\sigma_\infty$-almost every sequence 
$([s_p])\in\Omega_1(L,F)$, the estimate (\ref{e:abs1.1}) holds 
for $s=s_p$ and $p$ large enough.
\end{corollary}

\proof
We apply Theorem \ref{th_abstract_1} for 
$V=H^0_{(2)}(X,L^p\otimes F)^*$ and for $\Phi=\Phi_p$, 
where $\Phi_p$ is the Kodaira map \eqref{e:kod}. 
The first assertion is a direct consequence of 
Theorem \ref{th_abstract_1}. 
We prove now the second assertion. 
The hypothesis \eqref{e:abs1.2} on $\lambda_p/\log p$ 
and \eqref{e:abs1.0} guarantee that
\[\sum_{p=1}^\infty \sigma_p(E_p)
\leqslant c'\sum_{p=1}^\infty\frac{1}{p^{\,\delta}}<\infty\,\]
for some $c'>0$ and $\delta>1$. 
Hence the set 
\begin{align}\label{eq:3.6}
E=\big\{ ([s_p]) \in \Omega_1(L,F): [s_p]\in E_p 
\text{ for an infinite number of indices } p\big\}
\end{align}
satisfies $\sigma_\infty(E)=0$. Indeed, for every $N\geq 0$,
it is contained in the set
$$\big\{ ([s_p]) \in \Omega_1(L,F): [s_p]\in E_p
\text{ for at least one  index } p\geq N\big\}$$
which is of $\sigma_\infty$-measure at most equal to
\begin{align}\label{eq:3.7}
\sum_{p=N}^\infty \sigma_p(E_p)
\leq c' \sum_{p=N}^\infty\frac{1}{p^{\,\delta}}=O(N^{1-\delta}).
\end{align}
Therefore, the second assertion of the corollary follows.
\endproof
%--------------------
We easily deduce from Corollary \ref{cor_abstract} the following.
%--------------------
\begin{corollary} \label{cor_abstract1}
Let $(X,\omega_X)$ be a compact K\"ahler manifold of 
dimension $n$ and let $(L,h^L)$ 
be a singular Hermitian holomorphic line bundle on $X$. 
Let $(F,h^F)$ be a 
holomorphic line bundle with a smooth Hermitian metric. 
Let $c=c(X,L,F)$ be the constant given by Corollary 
\ref{cor_abstract} and let $(\lambda_p)$ be a sequence 
of positive numbers satisfying 
\begin{equation}\label{e:abs1.3}
\liminf_{p\to\infty} \frac{\lambda_p}{\log p}>(2n+1)c\,,\;\;
\lim_{p\to\infty} \frac{\lambda_p}{p}=0\,.
\end{equation}
Let $U\subset X$ be an open set. Assume that 
$({\frac1p}\omega_p)$ converges to a current $\Theta$ 
in $U$ with speed $(c_p)$\,. Then for $\sigma_\infty$-almost 
every sequence $([s_p])\in\Omega_1(L,F)$, 
${(\frac1p}[\Div(s_p)])$ converges to $\Theta$ on $U$ with 
speed $\big(c_p+\frac{\lambda_p}{p}\big)$ as $p\to\infty$. 
\end{corollary}

We consider now products of projective spaces. 
Let $\pi_i:\P(V^*)^k\to \P(V^*)$, $i=1,\ldots, k$, 
be the canonical projections from the multi-projective space 
$\P(V^*)^k\simeq (\P^d)^k$ onto its factors. 
As usual we denote by $\omega_\FS$ the Fubini-Study form 
on $\P(V^*)$. 
Consider the K\"ahler form and volume form on $\P(V^*)^k$, 
\begin{equation}\label{MP}
\omega_\MP:=c_{d,k} \sum_{i=1}^k \pi_i^*(\omega_\FS)\,,\:\:
\sigma_\MP:=\omega_\MP^{kd}\,,
\end{equation}
where $c_{d,k}$ is the positive constant so that the volume form 
$\sigma_\MP$ defines a probability measure. 
The constant $c_{d,k}$ is given by the formula
\begin{align}\label{eq:3.9}
(c_{d,k})^{-dk}=
\Big(\begin{array}{c}
dk \\ 
d
\end{array} \Big)
\Big(\begin{array}{c}
dk-d \\ 
d
\end{array} \Big)
\cdots
\Big(\begin{array}{c}
2d \\ 
d
\end{array} \Big)=\frac{(dk)!}{(d!)^k}
\cdot
\end{align}
Using Stirling's formula $n!\simeq \sqrt{2\pi n} n^n e^{-n}$, 
one can show that $c_{d,k}$ is smaller than 1 and larger than 
a strictly positive constant depending only on $k$. 
The measure $\sigma_\MP$ is the Haar measure associated 
with the natural action of 
the unitary group on the factors of $\P(V^*)^k$.

We give now the proof of Theorem \ref{th_abstract_high}. 
Recall that a {\it quasi-psh} function is locally the difference 
between a 
psh function and a smooth function. A quasi-psh function $u$ on 
$\P(V^*)^k$ is {\it $\omega_\MP$-psh} if it satisfies 
$\ddc u\geq -\omega_\MP$, 
i.e., $\ddc u+\omega_\MP$ is a positive current.  
We need the following result from \cite[Proposition A.9]{DS06b}. 

\begin{lemma} \label{lemma_DS}
There are $c>0$, $\alpha>0$ and $m>0$ depending only 
on $k$ such that 
if $u$ is an $\omega_\MP$-psh function on $\P(V^*)^k$ with  
$\int u d\sigma_\MP=0$, then 
\begin{align}\label{eq:3.11}
u \leq c(1+\log d) \quad \mbox{and}\quad \sigma_\MP\{u<-t\} 
\leq c\, d^me^{-\alpha t} \quad \mbox{for }\, \, t\geq 0.
\end{align}
\end{lemma}

\begin{lemma} \label{lemma_key}
Let $\Sigma$ be a closed subset of $\P(V^*)^k$ and  
let $u$ be an $L^1$ function which 
is continuous on $\P(V^*)^k\setminus \Sigma$. 
Let $\gamma$ be a positive constant. 
Suppose there is a positive closed $(1,1)$-current $S$ of 
mass $1$ on  $\P(V^*)^k$ such that 
$-S\leqslant \ddc u\leqslant S$ 
and $\int ud\sigma_\MP=0$.  Then, there are 
$c>0$, $\alpha>0$, $m>0$ depending only on $k$ and a Borel set 
$E'\subset \P(V^*)^k$ depending only on $S$ and $\gamma$ 
such that
\begin{align}\label{eq:3.12}
\sigma_\MP(E')\leqslant c\, d^me^{-\alpha\gamma}\quad 
\mbox{and}\quad |u(a)|\leqslant \gamma \quad \mbox{for}
\quad a\not\in \Sigma\cup E'.
\end{align}
\end{lemma}
\proof
By K\"unneth's formula, the cohomology group 
$H^{1,1}(\P(V^*)^k,\R)$ is generated 
by the classes of $\pi_i^*(\omega_\FS)$ with $i=1,\ldots,k$. 
Therefore, there are $\lambda_i\geq0$ such that the class 
$\{S\}$ of $S$ is equal to 
$\sum \lambda_i \{\pi_i^*(\omega_\FS)\}$. 
The mass of $S$ can be computed cohomologically. 
If we identify the top bi-degree cohomology group 
$H^{kd,kd}(\P(V^*)^k,\R)$ with $\R$ in the canonical way, 
this mass is equal to the cup product 
$\{S\}\smallsmile \{\omega_\MP\}^{kd-1}$ 
and then a direct computation gives 
\begin{align}\label{eq:3.14}
\sum_{i=1}^k \lambda_i (c_{d,k})^{kd-1} \Big(\begin{array}{c}
dk-1 \\ 
d-1
\end{array} \Big)
\Big(\begin{array}{c}
dk-d \\ 
d
\end{array} \Big)
\cdots
\Big(\begin{array}{c}
2d \\ 
d
\end{array} \Big)
=\sum_{i=1}^k \lambda_i (c_{d,k})^{-1}k^{-1}.
\end{align}
We used here that $\{\omega_\FS\}^d=1$ in 
$H^{d,d}(\P(V^*),\R)\simeq \R$. 
Since $c_{d,k}\leq 1$ and $S$ is of mass $1$, 
we deduce that  $\lambda_i\leq k$.

By the $\ddc$-lemma \cite[Lemma\,1.5.1]{MM07}, 
there is a unique quasi-psh function $v$ such that 
\begin{align}\label{eq:3.15}
\ddc v= S-\sum_{i=1}^k \lambda_i \pi_i^*(\omega_\FS) 
\quad \mbox{and} 
\quad \int v d\sigma_\MP=0.
\end{align}
We have $\ddc v +\lambda\omega_\MP\geq S$ for some constant 
$\lambda>0$ depending only on $k$. 
Define $w:=\lambda^{-1}(u+v)$. We have 
$\ddc w\geq  -\omega_\MP$. 
Since $u$ is continuous outside $\Sigma$, the latter property
implies that $w$ is equal outside $\Sigma$ to a quasi-psh function. 
We still denote this quasi-psh function by $w$. Applying
Lemma \ref{lemma_DS} to $w$ instead of $u$, we obtain that
\begin{align}\label{eq:3.16}
u=\lambda w-v\leqslant c\lambda(1+\log d) -v.
\end{align}

Let $E'$ denote the set $\{v<-\gamma+c\lambda (1+\log d)\}$ 
which does not 
depend on $u$. Clearly, $u\leqslant \gamma$ outside 
$\Sigma\cup E'$. 
The same property applied to $-u$ implies that
$|u|\leqslant \gamma$ outside $\Sigma\cup E'$. 
It remains to bound the size of $E'$.
Lemma \ref{lemma_DS} applied to $\lambda^{-1} v$ yields
\begin{align}\label{eq:3.17}
\sigma_\MP(E')\leq c\, d^m \exp\big(-\alpha\lambda^{-1}\gamma
+c\alpha(1+\log d)\big).
\end{align}
This is the desired inequality for (other) suitable constants 
$c,\alpha$ and $m$.
\endproof

\begin{proof}[End of the proof of Theorem \ref{th_abstract_high}]  
Let $\Phi:X\dashrightarrow\P(V)$ be a meromorphic map and let
$\Gamma\subset X\times \P(V)$ its graph. We define 
\begin{equation}\label{e:incid}
\widetilde{X}=\big\{(x,\xi)\in X\times \P(V^*)^k:
\text{$\exists v\in\P(V)$ 
such that $(x,v)\in\Gamma$, $v\in H_\xi$}\big\}.
\end{equation}
Recall that for $\xi=(\xi_1,\ldots,\xi_k)\in \P(V^*)^k$ we denote 
$H_\xi=H_{\xi_1}\cap\ldots\cap H_{\xi_k}$
the intersection of the hyperplanes $H_{\xi_i}$ in $\P(V )$.
%$\widetilde X$ denote the closure of the set of points 
%$(x,\xi)\in X\times \P(V^*)^k$ 
%such that $\Phi(x)$ is well-defined and belongs to $H_{\xi_i}$ 
%for every $i$, where $\xi=(\xi_1,\ldots,\xi_k)$. 
The set $\widetilde{X}$ is a compact analytic subset
in $X\times \P(V^*)^k$,
of dimension $n+(d-1)k$.
Let $\Pi_1$ and $\Pi_2$ denote the natural projections from 
$\widetilde X$ onto $X$ and $\P(V^{*})^k$ respectively.
\begin{lemma}\label{L:exc}
Let $\Sigma\subset \P(V^*)^k$ be the set of points $\xi$ 
such that 
$ \widetilde X\cap \Pi_2^{-1}(\xi)\not=\emptyset$ 
and one of the following properties holds:
\begin{enumerate}
\item[(a)] $\dim H_\xi>d-k$;
\item[(b)] $\dim \widetilde X\cap \Pi_2^{-1}(\xi)>n-k$; 
\item[(c)] $\dim H_\xi=d-k$, $\dim \widetilde X\cap \Pi_2^{-1}(\xi)
=n-k$ 
but the last intersection is not transversal at a generic point. 
\end{enumerate}
Then $\Sigma$ is contained in a proper analytic subset of 
$\P(V^*)^k$.  
\end{lemma}
\proof
If $\Pi_2$ is not surjective, the lemma is clear because $\Sigma$ 
is contained in $\Pi_2(\widetilde X)$ which is a proper analytic 
subset of $\P(V^*)^k$. Assume that $\Pi_2$ is surjective. 
So $ \widetilde X\cap \Pi_2^{-1}(\xi)\not=\emptyset$ 
for every $\xi$. 
Observe that the set $\Sigma_1$ of $\xi$ satisfying (a)
is a proper analytic subset of $\P(V^*)^k$. 
Thus, we only consider parameters $\xi$ outside $\Sigma_1$. 

Let $\tau:\widehat X\to\widetilde X$ be a singularity resolution 
for $\widetilde X$ and define $\widehat\Pi_2:=\Pi_2\circ\tau$. 
The last map is a holomorphic surjective map between compact 
complex manifolds. 
So by Bertini-Sard type theorem, there is a proper analytic 
subset $\Sigma_2$ of $\P(V^*)^k$ such that $\widehat\Pi_2$
is a submersion outside $\widehat\Pi_2^{-1}(\Sigma_2)$.
Indeed, $\Sigma_{2}$ is the set of critical values of $\widehat\Pi_2$
which is analytic. Sard's theorem implies that it is a proper 
analytic subset of $\P(V^*)^k$.
It follows that for $\xi\not\in\Sigma_1\cup\Sigma_2$
the fiber $\widehat\Pi_2^{-1}(\xi)$ has dimension $n-k$,
i.e., the minimal dimension for the fibers of $\widehat \Pi_2$. 
So $\Pi_2^{-1}(\xi)$, which is the image of 
$\widehat\Pi_2^{-1}(\xi)$ by $\tau$, is also of minimal dimension 
$n-k$. Therefore, such parameters $\xi$ do not satisfy (b). 

Let $E$ denote the exceptional analytic subset in $\widehat X$,
i.e., the pull-back of the singularities of $\widetilde X$ by 
$\tau$. Since $\dim E<\dim \widehat X$, arguing as above, 
we obtain a proper analytic subset $\Sigma_3$ of $\P(V^*)^k$ 
such that for $\xi$ outside $\Sigma_3$, the dimension of 
$E\cap \widehat \Pi_2^{-1}(\xi)$ is at most equal to $n-k-1$. 
Since $\tau$ is locally bi-holomorphic outside $E$, 
for $\xi\not\in \Sigma_1\cup\Sigma_2\cup\Sigma_3$, 
the intersection $\widetilde X\cap\Pi_2^{-1}(\xi)$ is 
transverse outside the image by $\tau$ of 
$E\cap \widehat\Pi_2^{-1}(\xi)$, which is of dimension 
at most $n-k-1$. Such parameters $\xi$ do not satisfy (c).
The lemma follows. 
\endproof

From now on, we only consider  $\xi\in\P(V^*)^k\setminus\Sigma$, 
where $\Sigma$ is defined in Lemma \ref{L:exc}. 
The current $[(\Pi_2)^*(\xi)]$ 
is then well-defined and we have 
\begin{align}\label{eq:3.19}
\Phi^*[H_\xi]=(\Pi_1)_*\big([(\Pi_2)^*(\xi)]\big).
\end{align}

There exists $C>0$ such that for any test smooth real 
$(n-k,n-k)$-form $\varphi$ 
on $X$ with $\|\varphi\|_{\Cc^2}\leq C$ we have
%Consider a test smooth real $(n-k,n-k)$-form $\varphi$ on $X$ 
%with $\Cc^2$-norm bounded by a suitable small constant. We have 
%such that
\begin{equation}\label{e:bd}
-\omega_X^{n-k+1}\leq \ddc \varphi\leq \omega_X^{n-k+1}.
\end{equation}
Take such a $\varphi$ and define $v:=(\Pi_2)_*(\Pi_1)^*(\varphi)$. 
This is a function on $\P(V^{*})^k$ whose value at 
$\xi\in \P(V^{*})^k\setminus\Sigma$ 
is the integration of $(\Pi_1)^*(\varphi)$ on the fiber 
$\Pi_2^{-1}(\xi)$. So, we have
\begin{align}\label{eq:3.20}
v(\xi)=\langle \Phi^*[H_\xi],\varphi\rangle\, .
\end{align}
Hence, $v$ is continuous on $\P(V^{*})^k\setminus\Sigma$. 
Since the form $\omega_\FS$ on $\P(V)$ is the average of 
$[H_{\xi_i}]$ with respect to the measure 
$\sigma_\FS$ on $\xi_i\in \P(V^*)$, the average of $[H_\xi]$ 
with respect to the measure 
$\sigma_\MP$ on $\xi\in \P(V^*)^k$ is equal to $\omega_\FS^k$. 
Thus, the mean value of $v$ is
\begin{align}\label{eq:3.21}
M_{v}:=\int vd\sigma_\MP
=\langle\Phi^*(\omega_\FS^k),\varphi\rangle\,.
\end{align}
So we need to prove that $|v- M_{v}|\leq \gamma m_{k-1}$ 
outside a set $E_{\gamma}$ of 
$\sigma_\MP$-measure less than $c\, d^{m}e^{-\gamma/c}$ 
which does not depend on $\varphi$. 
This implies Theorem \ref{th_abstract_high}.

Define 
\begin{align}\label{eq:3.23}
T:=(\Pi_2)_*(\Pi_1)^*(\omega_X^{n-k+1}).
\end{align}
This is a positive closed $(1,1)$-current on $\P(V^*)^k$ and 
we have, 
thanks to the above property \eqref{e:bd} of $\ddc \varphi$, that
\begin{align}\label{eq:3.24}
-T\leq \ddc v\leq T.
\end{align}
Let $\vartheta$ be the mass of $T$. We can apply 
Lemma \ref{lemma_key} to the function 
$u:=\vartheta^{-1}(v- M_{v})$, $S:=\vartheta^{-1}T$ and to 
$\vartheta^{-1} \gamma m_{k-1}$ instead of $\gamma$. 
We can take $E_{\gamma}=\Sigma\cup E'$ 
which does not depend on $\varphi$. Since $\Sigma$ is 
of measure 0, in order to get 
from Lemma \ref{lemma_key} the desired estimate on 
$\sigma_\MP(E_{\gamma})=\sigma_\MP(E')$, 
it is enough to show that $\vartheta$ is bounded above 
by $m_{k-1}$ times a constant which only depends on $k$. 

We have
\begin{align}\label{eq:3.25}
\|T\|=\big\langle (\Pi_2)_*(\Pi_1)^*(\omega_X^{n-k+1}), 
\omega_\MP^{kd-1}\big\rangle=
\big\langle \omega_X^{n-k+1}, (\Pi_1)_*(\Pi_2)^*
(\omega_\MP^{kd-1})\big\rangle.
\end{align}
Let $\widetilde{\P(V)}$ denote the set of points 
$(x,\xi)\in \P(V)\times \P(V^*)^k$ 
such that $x\in H_{\xi_i}$ for every $i$. 
Denote by $\Pi_1'$ and $\Pi_2'$ 
the natural projections from $\widetilde{\P(V)}$ onto $\P(V)$ 
and $\P(V^*)^k$. By construction, we have 
\begin{align}\label{eq:3.26}
(\Pi_1)_*(\Pi_2)^*(\omega_\MP^{kd-1})
= \Phi^*\big((\Pi_1')_*(\Pi_2')^*(\omega_\MP^{kd-1})\big).
\end{align}
By definition of $m_{k-1}$, it is enough to check that 
$(\Pi_1')_*(\Pi_2')^*(\omega_\MP^{kd-1})$ is bounded by 
$\omega_\FS^{k-1}$ 
times a constant depending only on $k$.

We obtain with a direct computation
\begin{align}\label{eq:3.27}
\omega_\MP^{kd-1}  = c_{d,k}^{kd-1} \Big(\begin{array}{c}
dk-1 \\ 
d-1
\end{array} \Big)
\Big(\begin{array}{c}
dk-d \\ 
d
\end{array} \Big)
\cdots
\Big(\begin{array}{c}
2d \\ 
d
\end{array} \Big)
\sum_{i=1}^k \Theta_i 
 =  (c_{d,k})^{-1} k^{-1}\sum_{i=1}^k \Theta_i ,
\end{align}
where 
\begin{align}\label{eq:3.28}
\Theta_i:=\pi_1^*(\omega_\FS^d)\wedge\ldots\wedge 
\pi_{i-1}^*(\omega_\FS^d)
\wedge \pi_{i}^*(\omega_\FS^{d-1})\wedge
\pi_{i+1}^*(\omega_\FS^d)\wedge\ldots 
\wedge \pi_k^*(\omega_\FS^d).
\end{align}
We will show that 
$(\Pi_1')_*(\Pi_2')^*(\Theta_i)=\omega_\FS^{k-1}$ 
and this implies the theorem.

For simplicity, assume that $i=1$. 
Since $(\Pi_1')_*(\Pi_2')^*(\Theta_1)$ 
is invariant under the action of the unitary group, 
it is equal to a constant times $\omega_\FS^{k-1}$. 
So we only have to check that the constant is 1 or equivalently 
the mass of $(\Pi_1')_*(\Pi_2')^*(\Theta_1)$ is 1. 
Recall that the mass of a positive 
closed current depends only on its cohomology class. 
Therefore, in the definition of $\Theta_1$, we can replace 
$\omega_\FS^{d-1}$ with the current of integration 
on a generic projective line $\ell$ and each $\omega_\FS^d$ 
with the Dirac mass 
of a generic point, say $\xi_j$, for $j=2,\ldots,k$. 
The current $\Theta_1$ is in the same cohomology class as 
the current of integration on 
$$\ell\times \{\xi_2\} \times\cdots\times\{\xi_k\}$$
that we denote by $\Theta_1'$. 

It is not difficult to see that $(\Pi_1')_*(\Pi_2')^*(\Theta_1')$ 
is the current of 
integration on the projective subspace  
$H_{\xi_2}\cap H_{\xi_3}\cap\ldots\cap H_{\xi_k}$. 
So it is clear that its mass is equal to 1. 
This completes the proof of the Theorem \ref{th_abstract_high}. 
\end{proof}

The following property of the constants $m_k$ is useful.

\begin{lemma} \label{le:mk}
There is $c>0$ depending only on $(X,\omega_X)$ such that 
$m_k\leq c\, m_1^k$
for $1\leq k\leq n$. 
\end{lemma}
\proof
Observe that by Lemma \ref{lemma_pullback},
the currents $\Phi^*(\omega_\FS^k)$ are given 
by $L^1$ forms and they are 
smooth on some Zariski open set $U$ of $X$ where $\Phi$ is 
holomorphic. Moreover, we have 
$\Phi^*(\omega_\FS^k)= \Phi^*(\omega_\FS)^k$ on $U$. 
Since $\Phi^*(\omega_\FS^k)$ is given by an $L^1$ form, 
it has no mass outside $U$.

By \cite[Lemma 2.2]{DN}, there is  $C>0$ depending only on 
$(X,\omega_X)$ such that if $T$ and $S$ are positive closed 
currents on $X$ which are smooth in an open set $U$ 
then the mass $\|T\wedge S\|_U$ of $T\wedge S$ on $U$
is bounded by $C\|T\|\|S\|$. 
By Skoda's extension theorem \cite[Th\'eor\`eme\,1]{Skoda}, 
positive closed currents of finite mass 
can be  extended by 0 through analytic sets. 
So if $U$ is a Zariski open set, 
the form $T\wedge S$ extends by 0 to a positive closed current 
on $X$ with mass bounded by $C\|T\|\|S\|$. 
This allows us to apply inductively the mass estimate for 
$T\wedge S$ to the case of product of 
several positive closed currents.
 
Observe that $\Phi^*(\omega_\FS^k)
= \Phi^*(\omega_\FS)\wedge \Phi^*(\omega_\FS^{k-1})$ on $U$,
so, by induction on $k$, we deduce from 
the above discussion %properties of $\Phi^*(\omega_\FS^k)$ 
that $m_k\leq C^{k-1} m_1^k$. 
The lemma follows.
\endproof

In the case where $V=H^0_{(2)}(X,L^p\otimes F)^*$, 
we have $m_1=O(p)$ 
and therefore $m_k=O(p^{k})$.
This together with Theorem \ref{th_abstract_high} 
imply the following corollary. 
Consider the Kodaira map 
$\Phi_p:X\dashrightarrow\P\big(H^0_{(2)}(X,L^p\otimes F)^*\big)$ 
defined in \eqref{e:kod}. The pull-back 
$\Phi_p^*(\omega_\FS^k)$ of the current $\omega_\FS^k$ is 
given by an $L^1$ form equal to $\omega_p^k$ on a dense Zariski 
open set (here $\omega_p$ is the Fubini-Study current \eqref{fs}). 

\begin{corollary} \label{cor_abstract_2}
There are $c=c(X,L,F)>0$ and $m=m(X,L,F)>0$  depending only 
on $(X,\omega_X)$ and $c_1(L)$, $c_1(F)$, 
%$c_1(L,h^L)$, $c_1(F,h^F)$, 
with the following property. For any sequence 
$\lambda_p$\,, there are subsets $E_p$ of 
$\P(H_{(2)}^0(X,L^p\otimes F))^k$ such that for $p$ large enough
\begin{enumerate}
\item[(a)] $\sigma_p(E_p)\leqslant c \, p^{m} e^{-\lambda_p/c}$.
\item[(b)] For $S_p=([s_p^{(1)}],\ldots,[s_p^{(k)}])$ in 
$\P(H^0_{(2)}(X,L^p\otimes F))^k\setminus E_p$\,, 
we have
\begin{equation}\label{est_abs_2}
\Big\|\frac{1}{p^k} \big[s_p^{(1)}=\ldots=s_p^{(k)}=0\big] 
-\frac{1}{p^k}\Phi_p^*(\omega_\FS^k)\Big\|_{-2} 
\leqslant \frac{\lambda_p}{p}\cdot
\end{equation}
\end{enumerate}
In particular, when  
\begin{equation}\label{est_abs_21}
\liminf_{p\to\infty} \frac{\lambda_p}{\log p}>(m+1)c\,,
\end{equation}
for $\sigma_\infty$-almost every sequence 
$(S_p)\in\Omega_k(L,F)$, 
the above estimate holds for $p$ large enough. 
If $\frac{1}{p^k}\Phi_p^*(\omega_\FS^k)$ converge to a current 
$\Theta_k$ in some open set $U$ 
with speed $(c_p)$ as $p\to\infty$, 
then $\frac{1}{p^k}\big[s_p^{(1)}=\ldots=s_p^{(k)}=0\big] $ 
converge to $\Theta_k$ on $U$ with speed 
$\big(c_p+\lambda_p/p\big)$ as $p\to\infty$. 
\end{corollary}

Note that the constants in the corollary can be chosen 
independently of $k$ because $1\leq k\leq n=\dim X$. 
The corollary can be applied in the situation of
Corollaries \ref{cor_cv_holder} and \ref{cor_cv_holder2}.
%Remark \ref{remark_wedge} below. 
In that cases, we have $\Theta_k=c_1(L,h^L)^k$ on $U$. 

%%%%%%%%%%%%%%%%%%%%%%%%%%%%%%%%%%%
\section{Semi-positive curved metrics on big line bundles} 
\label{section_hyp}
%%%%%%%%%%%%%%%%%%%%%%%%%%%%%%%%%
Let $(X,\omega_X)$ be a compact K\"ahler manifold 
of dimension $n$. Let $(L,h^L)$ be a holomorphic 
line bundle endowed with a singular metric $h^L$.
Fix a smooth Hermitian metric $h^L_0$ on $L$ and let 
$\alpha=c_1(L,h^L_0)$ 
denote its first Chern form. We can write
\begin{equation}\label{e:1}
h^L=e^{-2\varphi} h^L_0\, , \text{ i.e., } 
|s|_{h^L}^2=|s|_{h^L_0}^2 e^{-2\varphi} 
\text{ for any section } s \text{ of } L\, ,
\end{equation}
where $\varphi$ is an $L^1$ function on $X$ with values
in $\R\cup\{\pm \infty\}$. 
We assume that the curvature of $h^L$ is semipositive, that is, 
$c_1(L,h^L)=\ddc\varphi+\alpha$ is a positive current. 
So the function $\varphi$ is $\alpha$-psh, i.e., $\varphi$
is quasi-psh and satisfies $\ddc\varphi\geq -\alpha$. Define 
\begin{equation}\label{e:1.1}
\omega:=c_1(L,h^L)=\ddc\varphi+\alpha.
\end{equation}
We also assume that the line bundle $L$ is big. So, there is a metric 
\begin{equation}\label{e:2}
\widetilde{h}^L=e^{-2\varphi'} h^L_0
\end{equation}
 such that 
$\omega':=\ddc\varphi'+\alpha\geq   \varepsilon\omega_X$ 
for some $\varepsilon>0$ (cf.\ \cite[Theorem\,2.3.30]{MM07}). 
Let $B_{p}$ be the Bergman function in (\ref{e:Bergfcn}) 
associated with $(L^p, (h^L)^{\otimes p})$. The function 
\begin{equation}\label{e:21}
\varphi_p:=\varphi+\frac1{2p}\log B_p
\end{equation}
is quasi-psh and by \eqref{e:Bergfcn2} satisfies
\begin{equation}\label{e:1.2}
\frac1p\,\omega_p=\ddc\varphi_p+\alpha,
\end{equation}
where $\omega_{p}$ are the Fubini-Study currents \eqref{fs}.
We call the functions $\varphi_{p}$ 
\emph{global Fubini-Study weights}.

%------------------------------
We we will use the $L^2$-estimates of 
Andreotti-Vesentini-H\"ormander for $\dbar$ in the following 
form (cf.\ \cite[Th\'eor\`eme 5.1]{D82}).
\begin{theorem}[$L^2$-estimates for 
 $\overline\partial$]\label{T:l2}
(i) Let $(X,\omega_{X})$ be a K\"ahler manifold of dimension $n$
which admits a complete K\"ahler metric. 
Let $(L,h^L)$ be a singular Hermitian holomorphic line bundle 
and let $\lambda:X\to[0,+\infty)$ 
be a continuous function such that 
$c_1(L,h^L)\geq\lambda\omega_{X}$.
Then for any form $g\in L_{n,1}^2(X,L,loc)$ satisfying
\begin{align}\label{eq:4.6}
{\overline\partial}g=0\,,\quad 
\int_X\lambda^{-1}|g|^2\,\omega_{X}^n<+\infty
\end{align}
there exists $u\in L_{n,0}^2(X,L)$ with 
$\overline\partial u=g$ and
\begin{align}\label{eq:4.7}
\int_X|u|^2\,\omega_{X}^n
\leq\int_X\lambda^{-1}|g|^2\,\omega_{X}^n\,.
\end{align}

(ii) Let $(X,\omega_X)$ be a complete K\"ahler manifold of 
dimension $n$ and 
let $(L,h^L)$ be a singular Hermitian line bundle. 
Assume that there exists $C>0$ such that
\[c_1(L,h^L)+c_1(K^{\ast}_X,h^{K^{\ast}_X})\geq  C\omega_X\]
where $h^{K^{\ast}_X}$ is the metric  induced  by $\omega_X$ 
on the anti-canonical bundle 
$K^{\ast}_X$. Then for any form $g\in L_{0,1}^2(X,L)$ 
satisfying  $\dbar g=0$ there exists $u\in L_{0,0}^2(X,L)$ with 
\begin{align}\label{eq:4.9}
\overline\partial u=g\,,\quad\int_X|u|^2\, \omega_X^n
\leqslant\frac1C\int_X|g|^2\, \omega_X^n\,.
\end{align}
\end{theorem}
%------------------------------
We will also need the following.
\begin{lemma} \label{lemma_max_psh}
Let $\psi$ be a negative psh function on a neighborhood of 
the unit ball $B$ in $\C^n$. Define 
\begin{align}\label{eq:4.10}
\psi'(z):=\sup_{B(z,\rho^4)} \psi,
\end{align}
where $B(z,\rho^4)$ denotes the ball of center $z$ and 
radius $\rho^4$. 
Then there is $c>0$ depending on $\psi$ 
such that for $\rho$ small enough
\begin{align}\label{eq:4.11}
\Big|\int_{B} \psi' dZ\Big| \geq  \Big|\int_{B} \psi dZ\Big|
- c\rho,
\end{align}
where $dZ$ denotes the Lebesgue measure on $\C^n$. 
\end{lemma}
\proof
In the last integral, we can replace $B$ by $B(0,1-2\rho^2)$ 
because by Cauchy-Schwarz inequality, 
the associated error is $O(\rho)$; we use here that psh functions 
are locally $L^2$-integrable. 
So, we have to prove that 
\begin{align}\label{eq:4.12}
\Big|\int_{B} \psi' dZ\Big| 
\geq  \Big|\int_{B(0,1-2\rho^2)} \psi dZ\Big|- c\rho.
\end{align}
It is enough to check for some (other) constant $c$ and for 
$\rho$ small enough that
\begin{align}\label{eq:4.14}
\Big|\int_{B} \psi' dZ\Big| 
\geq  (1-c\rho)\Big|\int_{B(0,1-2\rho^2)} \psi dZ\Big|.
\end{align}

\noindent
We claim that
\begin{align}\label{eq:4.15}
\rho^{-8n}\int_{B(z,\rho^4)}\psi' dZ
\leqslant (1-c\rho)\rho^{-4n} \int_{B(z,\rho^2)} \psi dZ .
\end{align}
The inequality can be rewritten as 
\begin{align}\label{eq:4.16}
\rho^{-8n}\int_{B(0,\rho^4)}\psi'(z+t) dZ(t)\leqslant 
(1-c\rho)\rho^{-4n} \int_{B(0,\rho^2)} \psi(z+t) dZ(t) .
\end{align}
Recall that $\psi$ and $\psi'$ are negative. 
Therefore, taking integrals in $z$ of both
 sides of the last inequality over $B(0,1-\rho^2)$ and using 
 Fubini's theorem for the variables $z$ and $t$, 
we obtain the desired inequality (\ref{eq:4.14}).
It remains to prove the claim.

Fix $x$ in $B(z,\rho^4)$. It is enough to check that 
\begin{align}\label{eq:4.17}
\psi'(x)\leqslant  
(1-c\rho)n!\pi^{-n}\rho^{-4n} \int_{B(z,\rho^2)} \psi dZ.
\end{align}
Note that the last expression is $1-c\rho$ times the average 
of $\psi$ on $B(z,\rho^2)$. 

By definition, there is $y\in B(z, 2\rho^4)$ such that
$\psi(y)=\psi'(x)$. So, there is a holomorphic automorphism $\tau$ 
of $B(z,\rho^2)$ such that $\tau(y)=z$ and 
$\|\tau-\id\|_{\Cc^1}=O(\rho)$ 
(cf. \cite[p. 25-28]{Rudin80}). 
Applying the sub-mean inequality to the psh function 
$\widetilde\psi:=\psi\circ \tau^{-1}$ at $z$ we have 
\begin{align}\label{eq:4.18}
\psi'(x)=\widetilde\psi(z)\leqslant 
n!\,\pi^{-n}\rho^{-4n} \int_{B(z,\rho^2)} \widetilde \psi dZ =
n! \,\pi^{-n}\rho^{-4n} \int_{B(z,\rho^2)} \psi \tau^*(dZ).
\end{align}
Observe that since $\|\tau-\id\|_{\Cc^1}=O(\rho)$, 
\begin{align}\label{eq:4.19}
\tau^*(dZ) \geq  (1-c\rho) dZ
\end{align}
for some $c>0$. The lemma follows.
\endproof

The following result gives us a situation where 
Corollary \ref{cor_abstract} applies.
It refines \cite[Theorem\,5.1]{CM11}, where it is shown that 
$\frac1p\log B_p\to 0$ in
$L^{1}(X,\omega_X^{n})$ for the Bergman kernel $B_{p}$ 
on powers $L^{p}$ of a big line bundle $L$ over 
a compact K\"ahler manifold $(X,\omega_X)$.
%------------------------------
\begin{theorem}\label{t4.1}
Let $(X,\omega_X)$ be a compact K\"ahler manifold of 
dimension $n$. Let $L$ be a big 
holomorphic line bundle and let $h^L$, $\widetilde{h}^L$ 
be singular Hermitian metrics on $L$ 
such that $c_1(L,h^L)\geq 0$ and 
$c_1(L,\widetilde{h}^L)\geq \varepsilon\omega_X$ 
for some $\varepsilon>0$. 
Assume there is $A>  0$ 
such that $h^L\leq A \,  \widetilde{h}^L$. Then 
\begin{equation}\label{el1}
\big\|\log B_p\big\|_{L^1(X)} =O(\log p)\,,\;\;p\to\infty\,.
%\big\|\log B_p\big\|_{L^1} \leqslant C\log p
\end{equation}
Hence $\dfrac{1}{p}\omega_p\to c_1(L,h^L)$ as $p\to\infty$
with speed $O(\frac1p\log p)$.
\end{theorem}
%------------------------------
\proof
Since we work only on $L$, we set in this proof for simplicity 
$h=h^L$, $\widetilde{h}=\widetilde{h}^L$.
Let $x\in X$ and $U_0\subset X$ 
be a coordinate neighborhood of $x$ on which there exists 
a holomorphic frame $e_{L}$ of $L$. 
Let $\psi$ be the psh weight of $h$ on $U_0$ relative to 
$e_{L}$, $|e_{L}|_{h}^2=e^{-2\psi}$. 
Likewise, let $\psi'$
be the psh weight of $\widetilde{h}$ on $U_0$ relative to 
$e_{L}$, $|e_{L}|_{\widetilde{h}}^2=e^{-2\psi'}$.
Multiplying the section $e_{L}$ with a constant allows
us to assume that $\psi\leq 0$. 
Fix $r_0>0$ so that the ball $V:=B(x,2r_0)$ of 
center $x$ and radius $2r_0$ is relatively 
compact in $U_0$ and let $U:=B(x,r_0)$.
By \cite[Theorem\,5.1]{CM11} and its proof (following \cite{D92}) 
there exists $C_1>0$ %, $p_0\in\mathbb{N}$ 
so that 
\begin{equation}\label{e:Bke}
\log B_p(z)\leq\log(C_1r^{-2n})
+2p\Big(\sup_{B(z,r)}\psi-\psi(z)\Big)
\end{equation}
holds for all $p\geq 1$, $0<r<r_0$ and $z\in U$ 
with $\psi(z)>-\infty$. 

\par Choose $r=1/p^4$. 
By applying Lemma \ref{lemma_max_psh} to $\psi$
we obtain from \eqref{e:Bke} that the integral on $U$ of 
the positive part of the right hand side of (\ref{e:Bke})
%the above inequality 
is smaller than $C_2\log p + C_2$ for some $C_2>0$\,.
Hence, in order to prove \eqref{el1} it remains to bound 
the negative part of $\log B_p$\,.

Multiplying $\widetilde{h}$ with a constant allows us 
to assume that $A=1$. 
So we have $h\leq \widetilde{h}$ and $\psi'\leq\psi$. 
Consider an integer $p_0$ (to be chosen momentarily). 
Write $L^p=L^{p-p_0}\otimes L^{p_0}$ and consider 
on $L^p$, $p>p_{0}$, the metric 
\begin{equation}\label{e:Hp}
H_p:=h^{\otimes (p-p_0)} \otimes \widetilde{h}^{\otimes p_0}\,,
\quad h_p:=h^{\otimes p}. 
\end{equation}  Then 
\begin{align}\label{eq:4.25}
c_1(L^p,H_p)=(p-p_0)c_1(L,h)
+p_0c_1(L,\widetilde{h})\geq  p_0\varepsilon \omega_X\,.
\end{align}
The weight of the metric $H_p$ with respect to the frame 
$e_{L}^{\otimes p}$ is 
$\Psi_p:=(p-p_0)\psi+p_0\psi'$ and we have 
$|e_{L}^{\otimes p}|_{H_p}^2=e^{-2\Psi_{p}}$.

\par Following \cite[Section 9]{D93b}, we proceed as in 
\cite[Theorem\,5.1]{CM11} 
to show that there exist $C_1>0$ and $p_0\in\mathbb{N}$ 
such that for all $p>p_0$ and all $z\in U$ with $\Psi_p(z)>-\infty$ 
there is a section 
$s_{z,p}\in H^0_{(2)}(X,L^p)$ with $s_{z,p}(z)\neq0$ and 
\begin{equation}\label{e:Bke1}
\int_X|s_{z,p}|^2_{H_p}\,\omega_X^n
\leq C_1|s_{z,p}(z)|^2_{H_p}\,.
\end{equation}

\par Let us prove the existence of $s_{z,p}$ as above. 
By the Ohsawa-Takegoshi 
extension theorem \cite{OT87} there exists $C'>0$ 
(depending only on $x$) such that 
for any $z\in U$ and any $p\in\N$ one can find a holomorphic 
function $v_{z,p}$ on $V$ with $v_{z,p}(z)\neq0$ and 
\begin{align}\label{eq:4.27}
\int_V|v_{z,p}|^2e^{-2\Psi_p}\omega_X^n
\leq C'|v_{z,p}(z)|^2e^{-2\Psi_p(z)}\,.
\end{align}
The function $v_{z,p}$ can be identified to a local section of 
$L^p$ satisfying an estimate similar to (\ref{e:Bke1}). 

\par We shall now solve the $\overline\partial$-equation with 
$L^2$-estimates in order to modify $v_{z,p}$ and get 
a global section $s_{z,p}$ of $L^p$ over $X$. 
Let $\theta\in\mathscr{C}^\infty(\mathbb R)$ be 
a cut-off function such that 
$0\leq\theta\leq1$, $\theta(t)=1$ for $|t|\leq\frac12$, 
$\theta(t)=0$ for $|t|\geq 1$.
Define the quasi-psh function $\varphi_z$ on $X$ by
\begin{align}\label{eq:4.29}
\varphi_z(y)=\begin{cases}n\theta\big(\tfrac{|y-z|}{r_0}\big)
\log\frac{|y-z|}{r_0}\,,\quad
\text{for $y\in U_0$}\,,\\
0,\quad\text{for $y\in X\setminus B(z,r_0)$}\,.
\end{cases}
\end{align}
We apply Theorem \ref{T:l2} (ii) for $(X,\omega_X)$ and 
$(L^p,H_p\,e^{-\varphi_z})$.
Note that there exists $C_3>0$ such that 
$dd^c\varphi_z\geq-C_3\omega_X$ for all $z\in U$.
We have
\begin{align}\label{eq:4.30}
c_1(L^p,H_p\,e^{-\varphi_z})=(p-p_0)\,c_1(L,h^L)
+p_0\,c_1(L,\widetilde{h}^L)+
dd^c\varphi_z\geq  (p_0\varepsilon-C_3) \omega_X\,.
\end{align}
Since $p_0$ is large enough, we have 
$(p_0\varepsilon-C_3)\omega_X+c_1(K_X^*,h^{K_X^*})
\geq  C_3\, \omega_X$.
%see Theorem \ref{T:l2} below for the notation. 
Thus, 
\begin{align}\label{eq:4.31}
c_1(L^p,H_p\,e^{-\varphi_z})+c_1(K^*_X,h^{K^*_X})
\geq  C_3\, \omega_X\,,\quad  \text{ for any } p\geq  p_0\,.
\end{align}

Consider the form 
\begin{align}\label{eq:4.32}
g\in L^2_{0,1}(X,L^p),\;
g=\overline\partial\big(v_{z,p}\,
\theta\big(\tfrac{|y-z|}{r_0}\big)e_{L}^{\otimes p}\big),
\end{align}
which vanishes outside $V$ and also on $B(z,r_0/2)$. 
By (\ref{eq:4.27}), (\ref{eq:4.32}) and $\Psi_p(z)>-\infty$,
we get 
\begin{align}\label{eq:4.34}\begin{split}
\int_X|g|^2_{H_p}\,e^{-2\varphi_z}\omega_X^n
&=\int_{V\setminus B(z,r_0/2)}|v_{z,p}|^2|
\overline\partial\theta(\tfrac{|y-z|}{r_0})|^{2}e^{-2\Psi_p}
e^{-2\varphi_z}\omega_X^n\\
&\leqslant C''\int_{V}|v_{z,p}|^2e^{-2\Psi_p}\omega_X^n
\leqslant C'' C' \, |v_{z,p}(z)|^2e^{-2\Psi_p(z)}<\infty,
\end{split}\end{align}
where $C''>0$ is a constant that depends only on $x$. 
By Theorem \ref{T:l2} (ii), (\ref{eq:4.31}) and (\ref{eq:4.34}), 
for each $p\geq p_0$ there exists $u\in L^2_{0,0}(X,L^p)$ 
such that $\overline\partial u =g$ and
\begin{align}\label{eq:4.33}
\int_X|u|^2_{H_p}\,e^{-2\varphi_z}\,\omega_X^n
\leqslant\frac{1}{C_3}\int_X|g|^2_{H_p}\,
e^{-2\varphi_z}\omega_X^n \,.
\end{align}
%Here the second integral is finite since $\Psi_p(z)>-\infty$ and 

Since $g$ is smooth, $u$ is also smooth. 
Near $z$, $e^{-2\varphi_z(y)}=r_0^{2n}|y-z|^{-2n}$ is not
integrable, thus $u(z)=0$. Define 
\begin{align}\label{eq:4.35}
s_{z,p}:=v_{z,p}\,\theta\big(\tfrac{|y-z|}{r_0}\big)
e^{\otimes p}_{L}-u.
\end{align}
Then 
\begin{align}\label{eq:4.36}
\overline\partial s_{z,p}=0, \quad
s_{z,p}(z)=v_{z,p}(z)e^{\otimes p}_{L}(z)\neq0, \quad
s_{z,p}\in H^0_{(2)}(X,L^p).
\end{align}
Since $\varphi_z\leqslant0$ on $X$, by (\ref{eq:4.27}),
(\ref{eq:4.34}), (\ref{eq:4.33}) and (\ref{eq:4.35}), we get
\begin{eqnarray*}
\int_X|s_{z,p}|^2_{H_p}\,\omega_X^n&\leqslant&
2\left(\int_V|v_{z,p}|^2e^{-2\Psi_p}\omega_X^n
+\int_X|u|^2_{H_p}e^{-2\varphi_z}\,\omega_X^n\right)\\
&\leqslant&2C'\left(1+\frac{C''}{C_3}\right)|v_{z,p}(z)|^2
e^{-2\Psi_p(z)}=C_1|s_{z,p}(z)|^2_{H_p}\,,
\end{eqnarray*}
with a constant $C_1>0$ that depends only on $x$. 
This concludes the proof of \eqref{e:Bke1}. 

By dividing both sides of \eqref{e:Bke1} by a constant, 
we obtain the existence of sections
$s_{z,p}\in H^0(X,L^p)$, $p>p_0$, such that 
\begin{equation}\label{e:Bke2}
\int_X|s_{z,p}|^2_{H_p}\,\omega_X^n=1\,,
\quad |s_{z,p}(z)|^2_{H_p}\geq  \frac{1}{C_{1}}.
\end{equation}
Since $\widetilde{h}\geq  h$, the first property of \eqref{e:Bke2} 
and \eqref{e:Hp} imply
\begin{equation}\label{e:Bke3}
\int_X|s_{z,p}|^2_{h_p}\,\omega_X^n\leqslant 1\,.
\end{equation}
%The second property of \eqref{e:Bke2} and 
%the definition \eqref{e:Hp} of $H_p$ and
Then \eqref{e:1}, \eqref{e:2}, \eqref{e:Hp} and
the second property of \eqref{e:Bke2} yield
\begin{equation}\label{e:Bke4}
|s_{z,p}(z)|_{h_p}^2 \geq C_{1}^{-1} e^{2p_0(\psi'(z)-\psi(z))}
=C_{1}^{-1}  e^{2p_0(\varphi'(z)-\varphi(z))}.
\end{equation}
Recall now (see e.\,g., \cite[Lemma\,3.1]{CM11}) that
\begin{equation}\label{e:Bke5}
\begin{split}
B_p(x)&=\max\{|s(x)|^2_{h_p}:\,s\in H^0_{(2)}(X,L^p),\;
\|s\|_p=1\}\\
&=\max\{|s(x)|^2_{h_p}:\,s\in H^0_{(2)}(X,L^p),\;\|s\|_p\leq1\}\,.
\end{split}
\end{equation}
It follows from \eqref{e:Bke3}-\eqref{e:Bke5} that there exists 
$C_5>0$ such that
\begin{align}\label{eq:4.38}
\log B_p(z)\geq \log |s_{z,p}(z)|_{h_p}^2
\geq 2 p_0\big(\varphi'(z)-\varphi(z)\big)-C_5
=:\eta(z)\,,
\end{align}
where $\eta\in L^{1}(X,\omega_X^{n})$, $\eta\leq0$. 
Hence $\log B_p\geq\eta$ a.\,\!e.\ on $X$.  
%The last expression is an integrable negative function on $X$. 
The result follows.
\endproof

\begin{corollary}\label{cor_cv_holder2}
Let $(X,\omega_X)$ be a compact K\"ahler manifold of 
dimension $n$. Let $L$ be a big 
holomorphic line bundle and let $h^L$, $\widetilde{h}^L$ 
be as in Theorem \ref{t4.1}.
Let $U$ be an open subset of $X$. 

\noindent
(i) Assume that the global weight $\varphi'$ 
of $\widetilde{h}^L$ given by \eqref{e:2} is bounded on 
a neighborhood of $\overline U$. Then
\begin{equation}\label{e:weiconv}
\|\varphi_p-\varphi\|_{L^1(U)}=O\left(\frac1p\log p\right)\,,
\:\:p\to\infty,
\end{equation}
and for every $1\leq k\leq n$ we have
\begin{equation}\label{e:fsk}
\frac{1}{p^k}\omega_p^k\to c_1(L,h^L)^k\,,\:\:p\to\infty,
\:\:\text{on $U$.}
\end{equation} 
 (ii) Assume moreover, $\varphi$ is H\"older continuous on 
 a neighborhood of $\overline U$. Then
\begin{equation}\label{e5.12}
\|\varphi_p-\varphi\|_{U,\infty}=O\left(\frac1p{\log p}\right)\,,
\:\:p\to\infty\,,
\end{equation}
 and \eqref{e:fsk} holds with speed  
 $O\big(\frac1p\log p\big)$. 
 
Hence
for $\sigma_{\infty}$-almost every sequence 
$(S_p)\in(\Omega_k(L),$ $\sigma_\infty)$,
$S_p=([s_p^{(1)}],\ldots,[s_p^{(k)}])$,
\begin{equation}\label{e5.13}
%\frac{1}{p^k}[s_p^{(1)}=\ldots=s_p^{(k)}=0]
\frac{1}{p^k}\big[s_p^{(1)}=\ldots=s_p^{(k)}=0\big]
\to c_1(L,h^L)^k\,,\:\:p\to\infty,
\:\:\text{on $U$ with speed $O\left(\frac1p\log p\right)$.}
\end{equation}
%with speed $O\big(\frac1p\log p\big)$. 
\end{corollary}
%-----------------
\begin{proof}
Since $\varphi'$ is bounded on a neighborhood 
of $\overline U$, $\varphi$ is also bounded in that neighborhood. 
We see in the above proof that \eqref{e:weiconv} holds 
and $\varphi_p+c/p\geq \varphi$ for some $c>0$. 
On the set where $\varphi$ and $\varphi_{p}$ are locally bounded 
the wedge-products
$\omega^k$ and $\omega_p^k$ are well-defined for any 
$1\leq k\leq n$ by \eqref{e:1.1}, \eqref{e:1.2} 
and \cite{BT82}. Thus \eqref{e:fsk} holds.

Assume moreover that $\varphi$ is H\"older continuous on 
a neighborhood of $\overline U$. 
Observe that the function $\eta$  in (\ref{eq:4.38}) is bounded on 
$U$, thus, taking $r=1/p^\ell$ with $\ell$ large enough 
in \eqref{e:Bke},  yields \eqref {e5.12}. 
Finally, \eqref{e5.13} follows from Corollary \ref{cor_abstract_2}.
\end{proof}
Note that under the assumptions of 
Corollary \ref{cor_cv_holder2} (i) we do not obtain 
an estimate of the convergence speed in \eqref{e:fsk}. 
To get this, the
assumption of H\"older continuity in item (ii) is necessary.

%------------------------------
We can state a result similar to Theorem \ref{t4.1} 
in the case of adjoint line bundles $L^p\otimes K_X$. 
We do not suppose that the base manifold is compact, 
so the space of $L^2$ holomorphic
sections could be infinite dimensional. However, 
the definitions \eqref{e:Bergfcn}
and \eqref{fs1} of the Bergman kernel function 
and Fubini-Study currents carry over without change. 
Theorem \ref{t4.2} refines \cite[Theorem\,3.1]{CM13b}, 
where it is shown that $\frac1p\log B_p\to 0$ in 
$L^{1}(U,\omega_X^{n})$.
%------------------------------
\begin{theorem}\label{t4.2}
Let $(X,\omega_{X})$ be a K\"ahler manifold of dimension $n$ 
which admits a (possibly different) complete K\"ahler metric.
Let $L$ be a holomorphic line bundle
and let $h^L$ be a singular Hermitian metric on $L$ such that 
$c_1(L,h^L)\geq 0$. 
Let $U\subset X$ be a relatively compact open set such that 
$c_1(L,h^L)\geq \varepsilon\omega_X$ on a neighborhood 
of $\overline U$ for some $\varepsilon>0$. 
Let $B_p$ and $\omega_p$ be the Bergman kernel function and 
Fubini-Study current associated 
with $H^0_{(2)}(X,L^p\otimes K_X)$. Then 
\begin{equation}\label{el2}
\big\|\log B_p\big\|_{L^1(U)} =O(\log p)\,,\;\;p\to\infty\,.
\end{equation}
Hence $\dfrac{1}{p}\omega_p\to c_1(L,h^L)$ on $U$ 
as $p\to\infty$ with speed $O\big(\frac1p\log p\big)$.    
\end{theorem}
%------------------------------
\begin{proof}
%{\bf ???????Est ce que ce n'est pas mieux d'utiliser 
%Theorem 4.6 pour que la preuve soit plus claire ??????? }
The proof is similar to the proof of Theorem \ref{t4.1}, 
with some simplifications due to the fact
that we don't need an auxiliary metric $\widetilde{h}^L$. 
The K\"ahler metric $\omega_X$ 
induces a metric on the canonical line bundle $K_X$ 
that we denote by $h^{K_X}$. 
We denote by $h_p$ the metric induced by
$h^L$ and $h^{K_{X}}$ on $L^p\otimes K_{X}$.
Let $U'$ be a neighborhood of $\overline U$ 
on which  the hypothesis $c_1(L,h^L)\geq\varepsilon\omega_X$ 
holds. We let $x\in U$ and $U_{0}\subset U'$ be a coordinate 
neighborhood of $x$ on which there exists a
holomorphic frame $e_{L}$ of $L$ and $e'$ of $K_{X}$. Let $\psi$
be a psh weight of $h^L$.
Fix $r_0>0$ so that the ball $V:=B(x,2r_0)\Subset U_{0}$ 
and let $W:=B(x,r_0)$.

\par Following the arguments of \cite[Theorem 5.1]{CM11} 
(or, more precisely, \cite[Theorem 4.2]{CM13},
where forms with values in $L^{p}\otimes K_{X}$ are considered)
we show that there exist
$C=C(W)>0$ and $p_0=p_{0}(W)\in\mathbb{N}$ 
so that
\begin{equation}\label{e:Bke7}
-\log C\leq\,\log B_p(z)\leq\log(Cr^{-2n})+
2p\Big(\max_{B(z,r)}\psi-\psi(z)\Big)
\end{equation}
holds for all $p> p_{0}$, $0<r<r_0$ and $z\in W$ 
with $\psi(z)>-\infty$.

The right-hand side estimate follows as in \cite[Theorem 5.1]{CM11}; 
%as in the proof of Theorem \ref{t4.1} 
it holds for all $p$ and does not require the hypothesis that $X$
is compact.

We prove next the lower estimate from \eqref{e:Bke7}. 
We proceed like in the proof of \cite[Theorem 4.2]{CM13} 
to show that there exist $C_2=C_2(W)>0$, 
$p_0=p_{0}(W)\in\mathbb{N}$
such that for all $p>p_0$ and all $z\in W$ with $\psi(z)>-\infty$
there exists $s_{z,p}\in H_{(2)}^0(X,L^p\otimes K_{X})$ 
with $s_{z,p}(z)\neq0$ and
\begin{equation}\label{e:sp}
\|s_{z,p}\|_p^2\leq C_2|s_{z,p}(z)|^2_{h_p}\,,
\end{equation}
where $\|s\|_p$ is the $L^2$ norm defined in \eqref{herm_prod}.
This is done exactly as in \cite[Theorem 4.2]{CM13}; 
the main point is again
the Ohsawa-Takegoshi extension theorem and the solution of the
$\overline\partial$-equation by the $L^{2}$ method from 
Theorem \ref{T:l2} (i).
Observe that \eqref{e:Bke5} and \eqref{e:sp} 
yield the desired lower estimate
\begin{equation}\label{e:Bke8}
\log B_p(z)=\,\max_{\|s\|_p=1}\log|s(z)|^2_{h_p}
\geq-\log C_2\,,\;\; \text{Êfor } p> p_{0},\,z\in W
\text{ and } \psi(z)>-\infty.
\end{equation}
Since $U$ is relatively compact we can choose $C_{2}$ 
and $p_{0}$ such that $\log B_p\geq-\log C_2$
holds a.\,\!e.\ on $U$ for all $p> p_{0}$.
As in the proof of Theorem \ref{t4.1}, we use the estimate 
from above in \eqref{e:Bke7}
and Lemma \ref{lemma_max_psh} to show the existence of 
$C_1=C_1(U')>0$ such that for all $p\in\N^*$,
\begin{equation*}%\label{e:Bke71}
\int_{U}(\log B_p)\,\omega_{X}^{n}\leq C_1\log p + C_1.
\end{equation*}
This completes the proof of Theorem \ref{t4.2}.
\end{proof}
%---------------

\begin{proof}[Proof of Theorem \ref{th_eq_speed}]
Combining Theorem \ref{t4.1} and Corollary \ref{cor_abstract1} 
applied to the case where $(F,h^F)$ is the trivial  line bundle 
and $\lambda_p=(2n+2)c\log p$, we obtain item (i).
Theorem \ref{t4.2} and Corollary \ref{cor_abstract1}
for $(F,h^F)=(K_X,h^{K_X})$
and the same $\lambda_p$ as above yield item (ii).
\end{proof}
%%%%%%%%%%%%%%%%%%%%%%%%%%%%%%%%%%
\section{Approximation of H\"older continuous 
weights}\label{s:hoelder}
%%%%%%%%%%%%%%%%%%%%%%%%%%%%%%%%%%
In this section we prove Theorem \ref{th_cv_holder} 
and Corollary \ref{cor_cv_holder}.
\proof[Proof of Theorem \ref{th_cv_holder}]
The continuity can be deduced directly from 
the estimate \eqref{e5.1}. 
%(or from results by Berman \cite{Berman} by approximating 
%$\varphi$ uniformly by smooth functions). 

%Denote by $h_{\eq}=h_{0}e^{-\varphi_{\eq}}$ the equilibrium 
%metric on $L$ associated with $\varphi_{\eq}$. 
We prove now \eqref{e5.1}.   
Recall that  $h_p = (h^L)^{\otimes p}$ is
 the metric on $L^p$ (cf. \eqref{e:hp}).  
Write as above 
$L^p=L^{p-p_0}\otimes L^{p_0}$ with the metric 
$H_{e,p}:=(h^L_{\eq})^{\otimes (p-p_0)} 
\otimes (h^L_0)^{\otimes p_0}$. 
As in Section \ref{section_hyp} (see \eqref{e:Bke2}), 
given a point $x_0\in X$, there exists a neighborhood $U(x_0)$ 
and $C>0$ such that for any $z\in U(x_0)$, one can find
a holomorphic section $s_{z,p}\in H^0(X,L^p)$ satisfying
\begin{equation}\label{e:Bke2b}
\int_X |s_{z,p}|^2_{H_{e,p}} \omega_X^n\leq C\,,\quad
|s_{z,p}(z)|_{H_{e,p}}=1.
\end{equation}
Since $\varphi_{\eq}$ and $\varphi$ are bounded and 
$\varphi_{\eq}\leq \varphi$, we deduce from \eqref{e:Bke2b} that
there exists $C>0$ such that
\begin{equation}\label{e:Bke2c}
\int_X |s_{z,p}|^2_{h_p}\omega_X^n\leq C\,,\quad
|s_{z,p}(z)|_{h_p}\geq C^{-1} e^{p(\varphi_{\eq}-\varphi)}.
\end{equation}
It follows from \eqref{e:Bke5} and \eqref{e:Bke2c} that there 
exists $c>0$ such that  we have
\begin{align}\label{eq:5.2}
{\frac1{2p}}\log B_p\geq  \varphi_{\eq}-\varphi- \frac{c}{p}
\:\:\:\text{on $X$}.
\end{align}
Since $\varphi_p=\varphi+{\frac1{2p}}\log B_p$, we obtain that
\begin{align}\label{eq:5.3}
\varphi_p-\varphi_{\eq}\geq  -\frac{c}{p}\:\:\:\text{on $X$}.
\end{align}

The estimate from above for $\varphi_p-\varphi_{\eq}$ is obtained 
using the submean inequality. 
Since $\varphi_p$ is $\alpha$-psh, 
by \eqref{e_phie}, 
it is enough to show that 
$\varphi_p \leq  \varphi+ \frac{c\log p}{p}$ on $X$ 
which is equivalent to $B_p\leq  p^{2c}$ for some $c>0$. 
Fix a point $a$ in $X$. Consider an arbitrary holomorphic section
$s\in H^0(X,L^p)$ such that 
\begin{align}\label{eq:5.4}
\int_X |s|_{h_p}^2\omega_X^n=1.
\end{align}
By \eqref{e:Bke5}, we only have to check that 
\begin{align}\label{eq:5.5}
|s(a)|_{h_p}^2\leq  p^{2c}.
\end{align}

Fix local holomorphic coordinates $z$ around $a$ with $|z|\leq  1$ 
and a holomorphic frame of $L$ such that $s$ is 
represented by a holomorphic function 
$f$ and the metric $h^L$ is represented by $e^{-\psi}$
with $\psi$ is H\"older continuous and $\psi(0)=0$. 
So, we have for some $C,\alpha>0$
\begin{align}\label{eq:5.6}
|\psi(z)|\leq  C|z|^\alpha.
\end{align}
Since $s$ has unit $L^2$-norm, the integral
$$\int_{|z|\leq  p^{-1/\alpha}} |f(z)|^2 e^{-2Cp|z|^\alpha} dZ$$
is bounded by a constant independent of $p$.
It follows that the integral of $|f|^2$ on the ball 
$B(0,p^{-1/\alpha})$ is bounded, because the function 
$e^{2Cp|z|^\alpha}$ is bounded there.
Therefore, by the submean inequality, we get
\begin{align}\label{eq:5.8}
|s(a)|_{h_p}^2  =|f(0)|^2\leq C' p^{2n/\alpha}.
\end{align}
This completes the proof.
\endproof

\begin{proof}[Proof of Corollary \ref{cor_cv_holder}] 
    Theorem \ref{th_cv_holder} together with 
Corollary \ref{cor_abstract_2} applied to
\[\lambda_p=(m+2)c\log p\] imply immediately the result.
\end{proof}

\begin{example}
Let us discuss here the important example of  the line bundle 
$L=\mathcal{O}(1)$ over $X={\mathbb P}^n$. 
The global holomorphic sections of $L^p=:\mathcal{O}(p)$ 
are given by homogeneous polynomials 
of degree $p$ on ${\mathbb C}^{n+1}$:
\begin{equation}
H^0({\mathbb P}^n,\mathcal{O}(p))
\cong\big\{f\in\C[w_0,\ldots,w_n]: 
\text{$f$ homogeneous, $\deg f=p$}\big\}=:R_p.
\end{equation}
There exists a smooth metric $h_{\FS}=h^{\cO(1)}_{\FS}$ 
on $\cO(1)$ such that the Fubini-Study K\"ahler 
form on ${\mathbb P}^n$ is defined as the first Chern form
associated to $(\cO(1), h_{\FS})$, 
\begin{equation}
\omega_{\FS}=\frac{i}{2\pi}R^{\cO(1)}.
\end{equation}
Let $\met^+(\mathcal{O}(1))$ be the set of all
semipositively curved singular metrics on $\mathcal{O}(1)$. 
By \eqref{psh-met} we know that there exists a bijection
\begin{equation}\label{psh-met1}
PSH({\mathbb P}^n,\omega_{\FS})
\longrightarrow\met^+(\mathcal{O}(1))\,,\:\: 
\varphi\longmapsto h_{\varphi}=h_{\FS} e^{-2\varphi},
\end{equation}
and $c_1(\mathcal{O}(1),h_{\varphi})=\omega_{\FS}+dd^c\varphi$.
Moreover, $PSH({\mathbb P}^n,\omega_{\FS})$ is in 
one-to-one correspondence to 
the Lelong class $\mathcal L({\mathbb C}^n)$ 
of entire psh functions with logarithmic growth, 
\[\mathcal L({\mathbb C}^n)=\Big\{ \psi\in PSH(\C^n):
\text{ there is } C_{\psi} \in \R \,\text{Êsuch that } 
\:\psi(z)
\leq\tfrac12\log(1+|z|^2)+C_\psi\,  \text{Êfor } z\in\C^n\Big\},\]
and the map $\mathcal L({\mathbb C}^n)
\to PSH({\mathbb P}^n,\omega_{\FS})$
is given by $\psi\mapsto\varphi$ where 
\[
\varphi=\begin{cases}
\psi(w)-\frac12\log(1+|w|^2)\,,\quad &w\in\C^n,\\
\limsup\limits_{z\to w,z\in\C^n}\varphi(z)\,,\quad
&w\in\P^n\setminus\C^n.
\end{cases}
\]
Here we use the usual embeding of $\C^n$ in $\P^n$. 
Let $h\in\met^+(\mathcal{O}(1))$ and let 
$\varphi\in PSH({\mathbb P}^n,\omega_{\FS})$ such that
$h=h_{\FS} e^{-2\varphi}$. Then
\begin{equation}
H^0_{(2)}({\mathbb P}^n,\mathcal{O}(p))
=\Big\{f\in H^0({\mathbb P}^n,\mathcal{O}(p)): 
\int_{\P^n}|f|^2_{h^p_{\FS}}e^{-2p\varphi}\omega^n_{\FS}
<\infty\Big\}=:R_p(\varphi).
\end{equation}
We denote as usual by $\omega_p$ the Fubini-Study current 
associated with
$H^0_{(2)}({\mathbb P}^n,\mathcal{O}(p))$ by \eqref{fs} and 
let $\varphi_p$ be the Fubini-Study global weights \eqref{e:21}.
Note that if $\varphi$ is bounded, 
$R_p(\varphi)=R_p$ (as sets but in general not as Hilbert spaces).

We have the following immediate consequence of
Theorem \ref{th_eq_speed} and Corollary \ref{cor_cv_holder2}.
%---------------
\begin{corollary}\label{cor_cv_holder3}
(i) Let $\varphi\in PSH({\mathbb P}^n,\omega_{\FS})$.  
Assume there exists $\widetilde{\varphi}
\in PSH({\mathbb P}^n,\omega_{\FS})$
such that 
\[\varphi\geq\widetilde{\varphi}\:\:\text{and}\:\:
(1-\varepsilon)\omega_{\FS}+\ddc\widetilde{\varphi}\geq0\,,\:\: 
\text{for some $\varepsilon>0$}.
\]
Then for $\sigma_\infty$-almost every sequence 
$[s_p]\in \P\big(R_p(\varphi)\big)$ of homogeneous polynomials, 
%$[s_p]\in \P\big(H^0_{(2)}(\P^n,\mathcal{O}(p))\big)$, 
${(\frac1p}[\Div(s_p)])$ 
converges to $\omega_{\FS}+\ddc\varphi$ on $\P^n$ 
as $p\to\infty$ with speed $O\big(\frac1p\log p\big)$.

\noindent
(ii) Let $U$ be an open subset of $\P^n$. Assume that 
$\widetilde{\varphi}$ 
is bounded on a neighborhood of $\overline U$. Then
\begin{equation}\label{e:weiconv1}
\|\varphi_p-\varphi\|_{L^1(U)}=O\left(\frac1p\log p\right)\,,
\:\:p\to\infty,
\end{equation}
and for every $1\leq k\leq n$ we have (not necessarily with 
speed estimate),
\begin{equation}\label{e:fsk1}
\frac{1}{p^k}\omega_p^k\to (\omega_{\FS}+\ddc\varphi)^k\,,
\:\:p\to\infty, \:\:\text{on $U$.}
\end{equation} 
(iii) Assume moreover that $\varphi$ is H\"older continuous on 
 a neighborhood of $\overline U$. Then
\begin{equation}\label{e5.12a}
\|\varphi_p-\varphi\|_{U,\infty}=O\left(\frac1p{\log p}\right)\,,
\:\:p\to\infty\,,
\end{equation}
 and \eqref{e:fsk1} holds with speed  
 $O\big(\frac1p\log p\big)$. 
 
Hence
for $\sigma_{\infty}$-almost every sequence 
$([s_p^{(1)}],\ldots,[s_p^{(k)}])\in\P\big(R_p(\varphi)\big)^k$ 
of $k$-tuples of homogeneous polynomials we have as $p\to\infty$,
\begin{equation}\label{e5.13a}
\frac{1}{p^k}\big[s_p^{(1)}=\ldots=s_p^{(k)}=0\big]
\to (\omega_{\FS}+\ddc\varphi)^k\,,
\:\:\text{on $U$ with speed $O\left(\frac1p\log p\right)$.}
\end{equation}
\end{corollary}
%-----------------
Theorem \ref{th_cv_holder} and Corollary \ref{cor_cv_holder} 
imply the following.
%---------------
\begin{corollary}\label{cor_cv_holder4}
Let $\varphi$ be a H\"older continuous function on $\P^n$.
Then: 

(i) The equilibrium weight $\varphi_{\eq}$ is 
continuous on $\P^n$ and the global Fubini-Study
weights $\varphi_p$ given by \eqref{e:21} converge to 
$\varphi_{\eq}$ uniformly with speed $O\big(\frac1p{\log p}\big)$.

(ii) 
For any $1\leq k\leq n$ we have 
$\frac{1}{p^k}\omega_p^k\to\omega_{\eq}^k$ on $\P^n$ as 
$p\to\infty$ with speed $O\big(\frac1p\log p\big)$. 

(iii)  Let $1\leq k\leq n$.  
For $\sigma_{\infty}$-almost every sequence 
$([s_p^{(1)}],\ldots,[s_p^{(k)}])\in\P\big(R_p(\varphi)\big)^k$,
\begin{equation}\label{e5.14a}
\frac{1}{p^k}\big[s_p^{(1)}=\ldots=s_p^{(k)}=0\big]
\to \omega_{\eq}^k\,,
\:\:\text{on $\P^n$ with speed $O\left(\frac1p\log p\right)$.}
\end{equation}

\end{corollary}
\end{example}
%-----------------

\end{document}